\documentclass[12pt]{article}
\usepackage{amsthm,amsfonts,amssymb,amsmath}
\usepackage{stmaryrd}
\usepackage{cite,hyperref}
\usepackage{epsfig}
\usepackage{url}
\usepackage{xcolor,tikz}
\usetikzlibrary{positioning, arrows.meta,calc}
\usetikzlibrary{matrix}
\usepackage{multicol,graphicx}
\usepackage{fullpage}
\usepackage{bbm}
\usetikzlibrary{decorations}

\usepackage[shortlabels]{enumitem}

\numberwithin{equation}{section}

\newtheorem{thm}[equation]{Theorem}

\newtheorem{lem}[equation]{Lemma}
\newtheorem{prop}[equation]{Proposition}

\newtheorem{ques}[equation]{Question}
\newtheorem{prob}[equation]{Problem}

\theoremstyle{definition}
\newtheorem{defn}[equation]{Definition}

\newtheorem{obs}[equation]{Observation}

\newtheorem*{ack}{Acknowledgements}

\theoremstyle{remark}

\newtheorem{rem}[equation]{Remark}

\title{Forcing Quasirandomness in a Regular Tournament}
\author{Jonathan A. Noel\thanks{Department of Mathematics and Statistics, University of Victoria, Victoria, B.C., Canada.}$\text{ }^{,}$\thanks{E-mail: {\tt noelj@uvic.ca}. Research supported by NSERC Discovery Grant RGPIN-2021-02460 and NSERC Early Career Supplement DGECR-2021-00024 and a Start-Up Grant from the University of Victoria.} \and Arjun Ranganathan\thanks{Department of Mathematics, University College London, Gower Street, London WC1E 6BT, UK. E-mail: {\tt arjun.ranganathan.24@ucl.ac.uk}.} \and Lina M. Simbaqueba\footnotemark[1]$\text{ }^{,}$\thanks{E-mail: {\tt lmsimbaquebam@uvic.ca}. Research Supported by a Mitacs Globalink Graduate Fellowship.}}

\DeclareTextCompositeCommand{\v}{OT1}{l}{l\nobreak\hspace{-.1em}'}
\DeclareTextCompositeCommand{\v}{OT1}{t}{t\nobreak\hspace{-.1em}'\nobreak\hspace{-.15em}}

\DeclareMathOperator{\Aut}{Aut}
\DeclareMathOperator{\inj}{inj}

\DeclareMathOperator{\aut}{aut}

\begin{document}

\maketitle

\begin{abstract}

A tournament $H$ is said to \emph{force quasirandomness} if it has the property that a sequence $(T_n)_{n\in \mathbb{N}}$ of tournaments of increasing orders is quasirandom if and only if the homomorphism density of $H$ in $T_n$ tends to $(1/2)^{\binom{v(H)}{2}}$ as $n\to\infty$. It was recently shown that there is only one non-transitive tournament with this property. This is in contrast to the analogous problem for graphs, where there are numerous graphs that are known to force quasirandomness and the well known Forcing Conjecture suggests that there are many more. To obtain a richer family of characterizations of quasirandomness in tournaments, we propose a variant in which the tournaments $(T_n)_{n\in \mathbb{N}}$ are assumed to be ``nearly regular.'' We characterize the tournaments on at most 5 vertices which force quasirandomness under this stronger assumption.
\end{abstract}

\section{Introduction}
A combinatorial structure is said to be ``quasirandom'' if it shares various properties which hold with probability tending to 1 in a large random structure of the same kind. For example, a dense graph is quasirandom if, for any two ``large enough'' sets of vertices, the density of edges between them closely approximates the global edge density. Many seemingly different but formally equivalent characterizations of quasirandom graphs exist; see~\cite{Thomason87a,Thomason87b,Rodl86,ChungGrahamWilson89}.
For example, an equivalent definition of quasirandomness in dense graphs is that the number of closed walks of length four, i.e. homomorphisms from a $4$-cycle, approximately matches the expectation in a random graph with the same number of vertices and edges. The well-known Forcing Conjecture of Conlon, Fox and Sudakov~\cite{ConlonFoxSudakov10}, based on a question of Skokan and Thoma~\cite{SkokanThoma04}, asserts that quasirandomness in graphs can be characterized in terms of the number of homomorphisms of any fixed bipartite graph with at least one cycle. The notion of quasirandomness was first introduced by R\"odl~\cite{Rodl86}, Thomason~\cite{Thomason87a,Thomason87b} and Chung, Graham and Wilson~\cite{ChungGrahamWilson89}. Since then, it has continued to attract a great deal of attention~\cite{Gowers08,ChungGraham91a,ChungGraham90,Gowers06,KohayakawaRodlSkokan02,Chan+20,CrudeleDukesNoel23+,SkokanThoma04,ConlonFoxSudakov18,ConlonHanPersonSchacht12,Cooper04,KralPikhurko13,Griffiths13,DellamonicaRodl11,CoreglianoRazborov23,GrzesikKralPikhurko24,CooperKralLamaisonMohr22,Kurecka22,KralLeeNoel24+} and was an inspiration behind combinatorial limit theory~\cite{Lovasz12,LovaszSos08,Chung14}. 

Our focus in this paper is on obtaining new characterizations of quasirandomness in tournaments, as studied in~\cite{CoreglianoRazborov17,CoreglianoParenteSato19,Hancock+23,BucicLongShapiraSudakov21,KalyanasundaramShapira13,ChungGraham91b,Grzesik+23}; recall that a \emph{tournament} is a directed graph (i.e. a \emph{digraph}) with no loops and exactly one arc between every pair of distinct vertices. A \emph{homomorphism} from a digraph $H$ to a digraph $D$ is a map $f:V(H)\to V(D)$ such that $f(u)f(v)\in A(D)$ whenever $uv\in A(H)$, where $A(F)$ denotes the set of arcs of a digraph $F$. Let $\hom(H,D)$ denote the number of homomorphisms from $H$ to $D$ and define the \emph{homomorphism density} of $H$ in $D$ to be
\[t(H,D):=\frac{\hom(H,D)}{v(D)^{v(H)}}\]
where $v(F)$ denotes the number of vertices in a digraph $F$. Intuitively, $t(H,D)$ is the probability that a uniformly random function from $f:V(H)\to V(D)$ is a homomorphism. 


If $T_n$ is a uniformly random tournament on $n$ vertices, i.e., each arc is directed randomly in one of the two possible directions with equal probability independent of all other arcs, and $H$ is a tournament on $k$ vertices, then
\[\mathbb{E}(t(H,T_n)) = (1-o(1))(1/2)^{\binom{k}{2}}\]
where the asymptotics are as $n\to\infty$. Furthermore, an application of the Azuma--Hoeffding Inequality (see~\cite[Theorem~7.2.1]{AlonSpencer}) tells us that the random variable $t(H,T_n)$ is tightly concentrated around this expected value. This motivates the following definition. A sequence $(T_n)_{n\in \mathbb{N}}$ of tournaments is said to be \emph{quasirandom} if, for every $k\geq1$ and every $k$-vertex tournament $H$,
\[\lim_{n\to\infty}t(H,T_n) = (1/2)^{\binom{k}{2}}.\]
In other words, a sequence of tournaments is quasirandom if every finite ``pattern'' appears with nearly the same frequency in a tournament in this sequence as it would in a sequence of uniformly random tournaments. Note that every quasirandom sequence of tournaments must satisfy $v(T_n)\to\infty$ as $t(H,T_n)=0$ if $H$ is a tournament with $v(H)>v(T_n)$. An important idea in quasirandomness is that it can often be characterized in terms of the frequencies of only a small number of patterns. The following definition is studied in, e.g.,~\cite{Hancock+23,BucicLongShapiraSudakov21,CoreglianoRazborov17}.

\begin{defn}
A tournament $H$ with $k$ vertices is said to \emph{force quasirandomness} if any sequence $(T_n)_{n\in \mathbb{N}}$ with $v(T_n)\to \infty$ and $\lim_{n\to\infty}t(H,T_n)=(1/2)^{\binom{k}{2}}$ is quasirandom. 
\end{defn}


Loosely speaking, this means that if the density of $H$ in $T_n$ is asymptotically equal to its density in a random tournament, then this is true for any other tournament $H'$ as well. In contrast to many combinatorial structures, such as graphs, the class of tournaments which force quasirandomness is very restricted. The \emph{transitive tournament} on $k$ vertices, denoted $TT_k$, is the tournament obtained by listing the vertices in a row and directing all arcs to the right. The fact that $TT_k$ forces quasirandomness for all $k\geq4$ is a slight extension of~\cite[Exercise~10.44]{Lovasz79}, and was reproved in~\cite{CoreglianoRazborov17} using the flag algebra method. Coregliano, Parente and Sato~\cite{CoreglianoParenteSato19} also used flag algebras to obtain an example of a $5$-vertex non-transitive tournament that forces quasirandomness. This is tournament $H_{17}$ in the list of all tournaments on at most $5$ vertices provided in Figure~\ref{fig:smallTourns}. Buci\'c, Long, Shapira and Sudakov~\cite{BucicLongShapiraSudakov21} proved that there are no non-transitive tournaments on 7 or more vertices which force quasirandomness. Finally, Hancock et al.~\cite{Hancock+23} showed that there are no tournaments on at most 6 vertices which force quasirandomness, apart from transitive tournaments on $4,5$ or $6$ vertices and $H_{17}$, completing the characterization of tournaments which force quasirandomness.

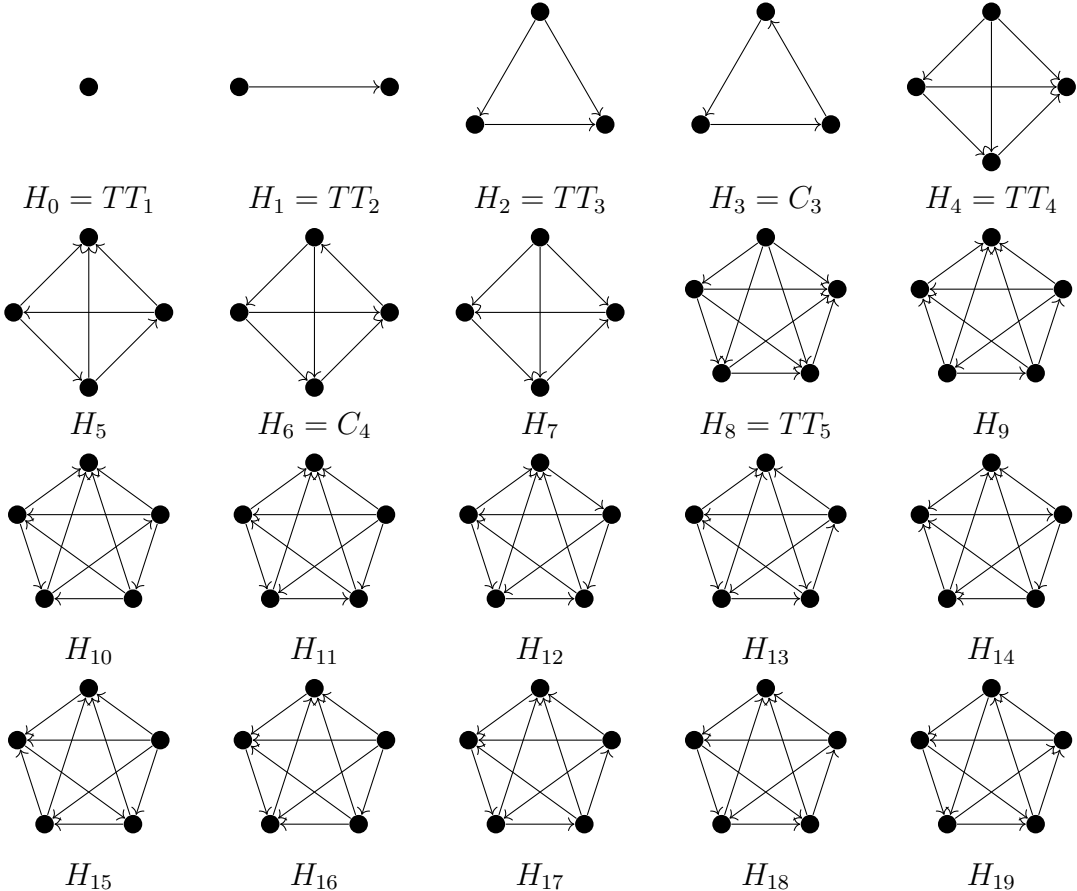
\begin{figure}[htbp]
    \begin{center}
    \begin{tikzpicture}[scale=1]
            \begin{scope}[xshift=0cm, yshift=0cm]
          \node[fill=black, circle, inner sep=0pt,minimum size=7pt] (1) {};
          \node (label) at (270:1.5cm) {$H_{0}=TT_1$};
        \end{scope}
        
        \begin{scope}[xshift=3cm, yshift=0.0cm]
          \node[fill=black, circle, inner sep=0pt,minimum size=7pt] at (180:1cm) (1) {};
          \node[fill=black, circle, inner sep=0pt,minimum size=7pt] at (0:1cm) (2) {};
          \node at (0,0) (3) {};
          \draw[->] (1) to (2);
          \node (label) at (270:1.5cm) {$H_{1}=TT_2$};
        \end{scope}
        
        \begin{scope}[xshift=6cm, yshift=0.0cm]
          \node[fill=black, circle, inner sep=0pt,minimum size=7pt] at (90:1cm) (1) {};
          \node[fill=black, circle, inner sep=0pt,minimum size=7pt] at (210:1cm) (2) {};
          \node[fill=black, circle, inner sep=0pt,minimum size=7pt] at (330:1cm) (3) {};
          \node at (270:0.7cm) (4) {};
          \draw[->] (1) to (2);
          \draw[->] (1) to (3);
          \draw[->] (2) to (3);
          \node (label) at (270:1.5cm) {$H_{2}=TT_3$};
        \end{scope}
        
        \begin{scope}[xshift=9cm, yshift=0.0cm]
          \node[fill=black, circle, inner sep=0pt,minimum size=7pt] at (90:1cm) (1) {};
          \node[fill=black, circle, inner sep=0pt,minimum size=7pt] at (210:1cm) (2) {};
          \node[fill=black, circle, inner sep=0pt,minimum size=7pt] at (330:1cm) (3) {};
          \node at (270:0.7cm) (4) {};
          \draw[->] (1) to (2);
          \draw[->] (2) to (3);
          \draw[->] (3) to (1);
          \node (label) at (270:1.5cm) {$H_{3}=C_3$};
        \end{scope}
        
        \begin{scope}[xshift=12cm, yshift=0.0cm]
          \node[fill=black, circle, inner sep=0pt,minimum size=7pt] at (90:1cm) (1) {};
          \node[fill=black, circle, inner sep=0pt,minimum size=7pt] at (180:1cm) (2) {};
          \node[fill=black, circle, inner sep=0pt,minimum size=7pt] at (270:1cm) (3) {};
          \node[fill=black, circle, inner sep=0pt,minimum size=7pt] at (0:1cm) (4) {};
          \node at (270:1cm) (label) {};
          \draw[->] (1) to (2);
          \draw[->] (1) to (3);
          \draw[->] (1) to (4);
          \draw[->] (2) to (3);
          \draw[->] (2) to (4);
          \draw[->] (3) to (4);
          \node (label) at (270:1.5cm) {$H_{4}=TT_4$};
        \end{scope}
        
        \begin{scope}[xshift=0.0cm, yshift=-3cm]
          \node[fill=black, circle, inner sep=0pt,minimum size=7pt] at (90:1cm) (1) {};
          \node[fill=black, circle, inner sep=0pt,minimum size=7pt] at (180:1cm) (2) {};
          \node[fill=black, circle, inner sep=0pt,minimum size=7pt] at (270:1cm) (3) {};
          \node[fill=black, circle, inner sep=0pt,minimum size=7pt] at (0:1cm) (4) {};
          \node at (270:1cm) (label) {};
          \draw[->] (2) to (1);
          \draw[->] (3) to (1);
          \draw[->] (4) to (1);
          \draw[->] (2) to (3);
          \draw[->] (3) to (4);
          \draw[->] (4) to (2);
          \node (label) at (270:1.5cm) {$H_5$};
        \end{scope}
        
        \begin{scope}[xshift=3.0cm, yshift=-3cm]
          \node[fill=black, circle, inner sep=0pt,minimum size=7pt] at (90:1cm) (1) {};
          \node[fill=black, circle, inner sep=0pt,minimum size=7pt] at (180:1cm) (2) {};
          \node[fill=black, circle, inner sep=0pt,minimum size=7pt] at (270:1cm) (3) {};
          \node[fill=black, circle, inner sep=0pt,minimum size=7pt] at (0:1cm) (4) {};
          \draw[->] (1) to (2);
          \draw[->] (2) to (3);
          \draw[->] (3) to (4);
          \draw[->] (4) to (1);
          \draw[->] (1) to (3);
          \draw[->] (2) to (4);
          \node (label) at (270:1.5cm) {$H_6=C_4$};
        \end{scope}
        
        \begin{scope}[xshift=6.0cm, yshift=-3cm]
          \node[fill=black, circle, inner sep=0pt,minimum size=7pt] at (90:1cm) (1) {};
          \node[fill=black, circle, inner sep=0pt,minimum size=7pt] at (180:1cm) (2) {};
          \node[fill=black, circle, inner sep=0pt,minimum size=7pt] at (270:1cm) (3) {};
          \node[fill=black, circle, inner sep=0pt,minimum size=7pt] at (0:1cm) (4) {};
          \node at (270:1cm) (label) {};
          \draw[->] (1) to (2);
          \draw[->] (1) to (3);
          \draw[->] (1) to (4);
          \draw[->] (2) to (3);
          \draw[->] (3) to (4);
          \draw[->] (4) to (2);
          \node (label) at (270:1.5cm) {$H_7$};
        \end{scope}
        
        \begin{scope}[xshift=9.0cm, yshift=-3cm]
          \node[fill=black, circle, inner sep=0pt,minimum size=7pt] at (90:1cm) (1) {};
          \node[fill=black, circle, inner sep=0pt,minimum size=7pt] at (162:1cm) (2) {};
          \node[fill=black, circle, inner sep=0pt,minimum size=7pt] at (234:1cm) (3) {};
          \node[fill=black, circle, inner sep=0pt,minimum size=7pt] at (306:1cm) (4) {};
          \node[fill=black, circle, inner sep=0pt,minimum size=7pt] at (18:1cm) (5) {};
          \node at (270:0.8cm) (label) {};
          \draw[->] (1) to (2);
          \draw[->] (1) to (3);
          \draw[->] (1) to (4);
          \draw[->] (1) to (5);
          \draw[->] (2) to (3);
          \draw[->] (2) to (4);
          \draw[->] (2) to (5);
          \draw[->] (3) to (4);
          \draw[->] (3) to (5);
          \draw[->] (4) to (5);
          \node (label) at (270:1.5cm) {$H_{8}=TT_5$};
        \end{scope}
        
        \begin{scope}[xshift=12.0cm, yshift=-3cm]
          \node[fill=black, circle, inner sep=0pt,minimum size=7pt] at (90:1cm) (1) {};
          \node[fill=black, circle, inner sep=0pt,minimum size=7pt] at (162:1cm) (2) {};
          \node[fill=black, circle, inner sep=0pt,minimum size=7pt] at (234:1cm) (3) {};
          \node[fill=black, circle, inner sep=0pt,minimum size=7pt] at (306:1cm) (4) {};
          \node[fill=black, circle, inner sep=0pt,minimum size=7pt] at (18:1cm) (5) {};
          \node at (270:0.8cm) (label) {};
          \draw[->] (2) to (1);
          \draw[->] (3) to (1);
          \draw[->] (4) to (1);
          \draw[->] (5) to (1);
          \draw[->] (3) to (2);
          \draw[->] (4) to (2);
          \draw[->] (5) to (2);
          \draw[->] (3) to (4);
          \draw[->] (5) to (3);
          \draw[->] (4) to (5);
          \node (label) at (270:1.5cm) {$H_9$};
        \end{scope}
        
        \begin{scope}[xshift=0.0cm, yshift=-6cm]
          \node[fill=black, circle, inner sep=0pt,minimum size=7pt] at (90:1cm) (1) {};
          \node[fill=black, circle, inner sep=0pt,minimum size=7pt] at (162:1cm) (2) {};
          \node[fill=black, circle, inner sep=0pt,minimum size=7pt] at (234:1cm) (3) {};
          \node[fill=black, circle, inner sep=0pt,minimum size=7pt] at (306:1cm) (4) {};
          \node[fill=black, circle, inner sep=0pt,minimum size=7pt] at (18:1cm) (5) {};
          \node at (270:0.8cm) (label) {};
          \draw[->] (2) to (1);
          \draw[->] (3) to (1);
          \draw[->] (4) to (1);
          \draw[->] (5) to (1);
          \draw[->] (2) to (3);
          \draw[->] (4) to (2);
          \draw[->] (5) to (2);
          \draw[->] (4) to (3);
          \draw[->] (3) to (5);
          \draw[->] (5) to (4);
          \node (label) at (270:1.5cm) {$H_{10}$};
        \end{scope}
        
        \begin{scope}[xshift=3.0cm, yshift=-6cm]
          \node[fill=black, circle, inner sep=0pt,minimum size=7pt] at (90:1cm) (1) {};
          \node[fill=black, circle, inner sep=0pt,minimum size=7pt] at (162:1cm) (2) {};
          \node[fill=black, circle, inner sep=0pt,minimum size=7pt] at (234:1cm) (3) {};
          \node[fill=black, circle, inner sep=0pt,minimum size=7pt] at (306:1cm) (4) {};
          \node[fill=black, circle, inner sep=0pt,minimum size=7pt] at (18:1cm) (5) {};
          \node at (270:0.8cm) (label) {};
          \draw[->] (2) to (1);
          \draw[->] (3) to (1);
          \draw[->] (4) to (1);
          \draw[->] (5) to (1);
          \draw[->] (2) to (3);
          \draw[->] (4) to (2);
          \draw[->] (5) to (2);
          \draw[->] (3) to (4);
          \draw[->] (5) to (3);
          \draw[->] (5) to (4);
          \node (label) at (270:1.5cm) {$H_{11}$};
        \end{scope}
        
        \begin{scope}[xshift=6cm, yshift=-6.0cm]
          \node[fill=black, circle, inner sep=0pt,minimum size=7pt] at (90:1cm) (1) {};
          \node[fill=black, circle, inner sep=0pt,minimum size=7pt] at (162:1cm) (2) {};
          \node[fill=black, circle, inner sep=0pt,minimum size=7pt] at (234:1cm) (3) {};
          \node[fill=black, circle, inner sep=0pt,minimum size=7pt] at (306:1cm) (4) {};
          \node[fill=black, circle, inner sep=0pt,minimum size=7pt] at (18:1cm) (5) {};
          \node at (270:0.8cm) (label) {};
          \draw[->] (2) to (1);
          \draw[->] (3) to (1);
          \draw[->] (4) to (1);
          \draw[->] (1) to (5);
          \draw[->] (2) to (3);
          \draw[->] (4) to (2);
          \draw[->] (5) to (2);
          \draw[->] (3) to (4);
          \draw[->] (5) to (3);
          \draw[->] (5) to (4);
          \node (label) at (270:1.5cm) {$H_{12}$};
        \end{scope}
        
        \begin{scope}[xshift=9cm, yshift=-6.0cm]
          \node[fill=black, circle, inner sep=0pt,minimum size=7pt] at (90:1cm) (1) {};
          \node[fill=black, circle, inner sep=0pt,minimum size=7pt] at (162:1cm) (2) {};
          \node[fill=black, circle, inner sep=0pt,minimum size=7pt] at (234:1cm) (3) {};
          \node[fill=black, circle, inner sep=0pt,minimum size=7pt] at (306:1cm) (4) {};
          \node[fill=black, circle, inner sep=0pt,minimum size=7pt] at (18:1cm) (5) {};
          \node at (270:0.8cm) (label) {};
          \draw[->] (2) to (1);
          \draw[->] (3) to (1);
          \draw[->] (1) to (4);
          \draw[->] (5) to (1);
          \draw[->] (2) to (3);
          \draw[->] (4) to (2);
          \draw[->] (5) to (2);
          \draw[->] (3) to (4);
          \draw[->] (5) to (3);
          \draw[->] (4) to (5);
          \node (label) at (270:1.5cm) {$H_{13}$};
        \end{scope}
        
        \begin{scope}[xshift=12cm, yshift=-6.0cm]
          \node[fill=black, circle, inner sep=0pt,minimum size=7pt] at (90:1cm) (1) {};
          \node[fill=black, circle, inner sep=0pt,minimum size=7pt] at (162:1cm) (2) {};
          \node[fill=black, circle, inner sep=0pt,minimum size=7pt] at (234:1cm) (3) {};
          \node[fill=black, circle, inner sep=0pt,minimum size=7pt] at (306:1cm) (4) {};
          \node[fill=black, circle, inner sep=0pt,minimum size=7pt] at (18:1cm) (5) {};
          \node at (270:0.8cm) (label) {};
          \draw[->] (1) to (2);
          \draw[->] (3) to (1);
          \draw[->] (4) to (1);
          \draw[->] (5) to (1);
          \draw[->] (3) to (2);
          \draw[->] (4) to (2);
          \draw[->] (2) to (5);
          \draw[->] (4) to (3);
          \draw[->] (5) to (3);
          \draw[->] (5) to (4);
          \node (label) at (270:1.5cm) {$H_{14}$};
        \end{scope}
        
        \begin{scope}[xshift=0cm, yshift=-9cm]
          \node[fill=black, circle, inner sep=0pt,minimum size=7pt] at (90:1cm) (1) {};
          \node[fill=black, circle, inner sep=0pt,minimum size=7pt] at (162:1cm) (2) {};
          \node[fill=black, circle, inner sep=0pt,minimum size=7pt] at (234:1cm) (3) {};
          \node[fill=black, circle, inner sep=0pt,minimum size=7pt] at (306:1cm) (4) {};
          \node[fill=black, circle, inner sep=0pt,minimum size=7pt] at (18:1cm) (5) {};
          \node at (270:0.8cm) (label) {};
          \draw[->] (1) to (2);
          \draw[->] (3) to (1);
          \draw[->] (4) to (1);
          \draw[->] (5) to (1);
          \draw[->] (3) to (2);
          \draw[->] (2) to (4);
          \draw[->] (5) to (2);
          \draw[->] (4) to (3);
          \draw[->] (5) to (3);
          \draw[->] (5) to (4);
          \node (label) at (270:1.5cm) {$H_{15}$};
        \end{scope}
        
        \begin{scope}[xshift=3cm, yshift=-9cm]
          \node[fill=black, circle, inner sep=0pt,minimum size=7pt] at (90:1cm) (1) {};
          \node[fill=black, circle, inner sep=0pt,minimum size=7pt] at (162:1cm) (2) {};
          \node[fill=black, circle, inner sep=0pt,minimum size=7pt] at (234:1cm) (3) {};
          \node[fill=black, circle, inner sep=0pt,minimum size=7pt] at (306:1cm) (4) {};
          \node[fill=black, circle, inner sep=0pt,minimum size=7pt] at (18:1cm) (5) {};
          \node at (270:0.8cm) (label) {};
          \draw[->] (1) to (2);
          \draw[->] (3) to (1);
          \draw[->] (4) to (1);
          \draw[->] (5) to (1);
          \draw[->] (2) to (3);
          \draw[->] (4) to (2);
          \draw[->] (5) to (2);
          \draw[->] (4) to (3);
          \draw[->] (5) to (3);
          \draw[->] (5) to (4);
          \node (label) at (270:1.5cm) {$H_{16}$};
        \end{scope}
        
        \begin{scope}[xshift=6cm, yshift=-9cm]
          \node[fill=black, circle, inner sep=0pt,minimum size=7pt] at (90:1cm) (1) {};
          \node[fill=black, circle, inner sep=0pt,minimum size=7pt] at (162:1cm) (2) {};
          \node[fill=black, circle, inner sep=0pt,minimum size=7pt] at (234:1cm) (3) {};
          \node[fill=black, circle, inner sep=0pt,minimum size=7pt] at (306:1cm) (4) {};
          \node[fill=black, circle, inner sep=0pt,minimum size=7pt] at (18:1cm) (5) {};
          \node at (270:0.8cm) (label) {};
          \draw[->] (1) to (2);
          \draw[->] (3) to (1);
          \draw[->] (4) to (1);
          \draw[->] (5) to (1);
          \draw[->] (2) to (3);
          \draw[->] (4) to (2);
          \draw[->] (5) to (2);
          \draw[->] (3) to (4);
          \draw[->] (5) to (3);
          \draw[->] (4) to (5);
          \node (label) at (270:1.5cm) {$H_{17}$};
        \end{scope}
        
        \begin{scope}[xshift=9cm, yshift=-9cm]
          \node[fill=black, circle, inner sep=0pt,minimum size=7pt] at (90:1cm) (1) {};
          \node[fill=black, circle, inner sep=0pt,minimum size=7pt] at (162:1cm) (2) {};
          \node[fill=black, circle, inner sep=0pt,minimum size=7pt] at (234:1cm) (3) {};
          \node[fill=black, circle, inner sep=0pt,minimum size=7pt] at (306:1cm) (4) {};
          \node[fill=black, circle, inner sep=0pt,minimum size=7pt] at (18:1cm) (5) {};
          \node at (270:0.8cm) (label) {};
          \draw[->] (1) to (2);
          \draw[->] (3) to (1);
          \draw[->] (4) to (1);
          \draw[->] (5) to (1);
          \draw[->] (2) to (3);
          \draw[->] (2) to (4);
          \draw[->] (5) to (2);
          \draw[->] (3) to (4);
          \draw[->] (5) to (3);
          \draw[->] (4) to (5);
          \node (label) at (270:1.5cm) {$H_{18}$};
        \end{scope}
        
        \begin{scope}[xshift=12cm, yshift=-9cm]
          \node[fill=black, circle, inner sep=0pt,minimum size=7pt] at (90:1cm) (1) {};
          \node[fill=black, circle, inner sep=0pt,minimum size=7pt] at (162:1cm) (2) {};
          \node[fill=black, circle, inner sep=0pt,minimum size=7pt] at (234:1cm) (3) {};
          \node[fill=black, circle, inner sep=0pt,minimum size=7pt] at (306:1cm) (4) {};
          \node[fill=black, circle, inner sep=0pt,minimum size=7pt] at (18:1cm) (5) {};
          \node at (270:0.8cm) (label) {};
          \draw[->] (1) to (2);
          \draw[->] (2) to (3);
          \draw[->] (3) to (4);
          \draw[->] (4) to (5);
          \draw[->] (5) to (1);
          \draw[->] (1) to (3);
          \draw[->] (3) to (5);
          \draw[->] (5) to (2);
          \draw[->] (2) to (4);
          \draw[->] (4) to (1);
          \node (label) at (270:1.5cm) {$H_{19}$};
        \end{scope}
    \end{tikzpicture}
    \end{center}
    \caption{The tournaments on at most $5$ vertices, up to isomorphism. There are 20 such tournaments labelled $H_0,\dots,H_{19}$. The tournaments $H_0,H_1,H_2,H_4$ and $H_8$ are transitive. The tournament $H_3$ is the cyclic tournament on 3 vertices and is also referred to as $C_3$. Similarly, $H_6$ is the unique $4$-vertex tournament containing a spanning cycle and is also referred to as $C_4$.}
    \label{fig:smallTourns}
\end{figure}

\begin{thm}[See~\cite{Lovasz79,BucicLongShapiraSudakov21,CoreglianoParenteSato19,Hancock+23,CoreglianoRazborov17}]
\label{th:qriff}
A tournament $H$ forces quasirandomness if and only if it is a transitive tournament on at least 4 vertices or is isomorphic to the tournament $H_{17}$ in Figure~\ref{fig:smallTourns}. 
\end{thm}

Our focus in this paper is on tournaments $H$ which force quasirandomness under the additional assumption that the sequence $(T_n)_{n\in \mathbb{N}}$ is \emph{nearly regular} meaning that, for every $\varepsilon>0$, there exists $n_0(\varepsilon)$ such that, for all $n\geq n_0(\varepsilon)$, all but at most $\varepsilon\cdot v(T_n)$ vertices of $T_n$ have out-degree between $\left(1/2 - \varepsilon\right)v(T_n)$ and $\left(1/2 + \varepsilon\right)v(T_n)$. Note that any nearly regular sequence of tournaments must satisfy $v(T_n)\to\infty$ since it is impossible for all vertices of $T_n$ to have out-degree greater than $\frac{v(T_n)-1}{2}$. Also, a sequence of uniformly random tournaments with the number of vertices tending to infinity is nearly regular with probability one; this follows from a simple application of the Chernoff bound. 

\begin{defn}
A tournament $H$ on $k$ vertices is said to \emph{force quasirandomness in regular tournaments} if any nearly regular sequence $(T_n)_{n\in \mathbb{N}}$ such that $\lim_{n\to\infty}t(H,T_n)=(1/2)^{\binom{k}{2}}$ is quasirandom.
\end{defn}

Our main result is the following theorem on tournaments on at most 5 vertices which force quasirandomness in regular tournaments. In contrast to Theorem~\ref{th:qriff}, the class of tournaments which force quasirandomness in regular tournaments seems as though it could be significantly broader than those which force quasirandomness in general (not necessarily regular) tournaments; see Questions~\ref{ques:infinitelyMany} and~\ref{ques:almostAll}. We note that a few cases of the next theorem follow from the earlier results of~\cite{CoreglianoParenteSato19,Lovasz79}; see Propositions~\ref{prop:alreadyKnownNeg},~\ref{prop:alreadyKnownPos} and~\ref{prop:C4}. 

\begin{thm}
\label{th:main}
Let $H_0,\dots,H_{19}$ be the list of all tournaments on at most $5$ vertices up to isomorphism, as in Figure~\ref{fig:smallTourns}. 
\begin{enumerate}[(i)]
    \item\label{thm:doNotForce} The tournaments $H_0,H_1,H_2,H_3,H_9,H_{12},H_{16},H_{18}$ and $H_{19}$ do not force quasirandomness in regular tournaments.
    \item\label{thm:doForce} The tournaments $H_4,H_5,H_6,H_7,H_8,H_{10},H_{11},H_{13},H_{14},H_{15}$ and $H_{17}$ force quasirandomness in regular tournaments.
\end{enumerate}
\end{thm}

The rest of the paper is organized as follows. In the next section, we translate the notions of homomorphism density, near regularity and quasirandomness into the language of ``limit objects'' for tournaments. Then, in Section~\ref{sec:reductions}, we derive some necessary conditions and sufficient conditions for the property of forcing quasirandomness in regular tournaments that will be used in our proofs. In Section~\ref{sec:constructions}, we present several families of tournament limit constructions which allow us to prove the assertion in Theorem~\ref{th:main}~\ref{thm:doNotForce}. Then, in Section~\ref{sec:flags}, we give a brief introduction to the flag algebra method and use it to prove several inequalities on linear combinations of homomorphism densities of tournaments on $3$ and $5$ vertices and show that equality holds only for a quasirandom sequence of tournaments. Theorem~\ref{th:main}~\ref{thm:doForce} is then derived from these results. We conclude the paper in Section~\ref{sec:concl} by discussing a few open problems. 

\section{Preliminaries}
\label{sec:prelims}


We begin with some standard notation and terminology related to tournaments. The \emph{adjacency matrix} of a tournament $T$ with vertices $v_1,\dots,v_n$ is the $n\times n$ matrix $A$ in which the entry on the $i$th row and $j$th column is equal to $1$ if $T$ contains an arc from $v_i$ to $v_j$ and $0$ otherwise. The \emph{out-degree} of a vertex $v\in V(T)$, denoted $d_T^+(v)$ is the number of arcs of $T$ that start at $v$ and the \emph{in-degree} $d_T^-(v)$ is the number of arcs of $T$ that end at $v$. In other words, the out-degree is the sum of the entries in the row corresponding to $v$ in the adjacency matrix of $T$, and the in-degree is the sum of the entries in the column corresponding to $v$. 

Our next goal is to introduce some tenets of the standard limit theory for tournaments; for a deeper treatment of combinatorial limits (focusing especially on graphs), see Lov\'asz~\cite{Lovasz12}. We note that, for our purposes, there is no formal mathematical advantage of dealing with limit objects as opposed to infinite sequences of tournaments. The main benefits are stylistic in nature; i.e. the limit setting allows us to avoid distracting the reader by repeatedly mentioning that ``$n$ is large'' and keeping track of floors, ceilings, and lower order asymptotic terms which are technically necessary but ultimately immaterial.

Following the terminology\footnote{It seems that the community has not settled on a definitive name for limit objects of tournaments. Other sources use \emph{tourneyon}~\cite{ZhaoZhou20}, \emph{tournament limit}~\cite{Chan+20} or \emph{tournament kernel}~\cite{Thornblad18}; however, the term tournamenton seems like the most popular variant so far.} in~\cite{Hancock+23,Grzesik+23,SahSawhney23+,MaTang22,GrzesikKralLovaszVolec23}, a \emph{tournamenton} is a measurable function $W:[0,1]^2\to[0,1]$ with the property that $W(x,y)+W(y,x)=1$ for all $x,y\in[0,1]$. Essentially, a tournamenton is just a ``continuous generalization'' of the adjacency matrix of a tournament; the property $W(x,y)+W(y,x)=1$ is analogous to the fact that the adjacency matrix $A$ of a tournament $T$ satisfies $A_{i,j}+A_{j,i}=1$ for all $i\neq j$. Given a digraph $D$ with vertex set $V(D)=\{v_1,\dots,v_k\}$ and a tournamenton $W$, the \emph{homomorphism density} of $D$ in $W$ is defined to be the $k$-fold integral
\[t(D,W):=\int_0^1\cdots\int_0^1 \prod_{v_iv_j\in A(D)}W(x_i,x_j)dx_1\cdots dx_k.\]
Another way of viewing $t(D,W)$ is in terms of the following random procedure. Pick $k$ points $x_1,\dots,x_k$ from $[0,1]$ uniformly at random and independently from one another and let $\{v_1,\dots,v_k\}$ be a set of vertices. Then, for each $i\neq j$, add an arc from $v_i$ to $v_j$ with probability $W(x_i,x_j)$ and from $v_j$ to $v_i$ otherwise. The tournament $T$ resulting from this is known as a \emph{$W$-random tournament}. It is not hard to see that $t(D,W)$ is precisely the probability that all arcs of $D$ are contained in a $W$-random tournament $T$. 

A sequence $(T_n)_{n\in \mathbb{N}}$ of tournaments with $v(T_n)\to\infty$ is said to \emph{converge} to a tournamenton $W$ if $\lim_{n\to\infty}t(H,T_n)=t(H,W)$ for every tournament $H$. The following proposition is a standard extension of a fundamental theorem in graph limits of Lov\'asz and Szegedy~\cite{LovaszSzegedy06}. It is not new; similar ideas are used in, e.g.,~\cite[Subsection~2.2]{ChanGrzesikKralNoel20}.

\begin{prop}
\label{prop:limitsExist}
For every sequence $(T_n)_{n\in \mathbb{N}}$ of finite tournaments with $v(T_n)\to\infty$ there is a subsequence of $(T_n)_{n\in \mathbb{N}}$ that converges to a tournamenton $W$. On the other hand, for every tournamenton $W$, if $(T_n)_{n\in \mathbb{N}}$ is a sequence of $W$-random touranments with $v(T_n)\to\infty$, then $(T_n)_{n\in \mathbb{N}}$ converges to $W$ with probability $1$.
\end{prop}

It will sometimes be convenient to extend the notion of quasirandom forcing from tournaments to sets of formal linear combinations of tournaments. Let $\mathcal{T}$ denote the set of all finite tournaments and $\mathbb{R}[\mathcal{T}]$ be the set of all formal linear combinations of tournaments; i.e. expressions of the form $\sum_{i=1}^t\alpha_i \cdot H_i$ for $\alpha_1,\dots,\alpha_t\in\mathbb{R}$ and $H_1,\dots,H_t\in\mathcal{T}$. We regard a single tournament $H$ as being the same as the formal linear combination $1\cdot H$. Given such a formal linear combination $H=\sum_{i=1}^t\alpha_i \cdot H_i$ and a tournament $T$, define $t(H,T) = \sum_{i=1}^t\alpha_i\cdot t(H_i,T)$ and let $t(H,1/2)=\sum_{i=1}^t\alpha_i \cdot (1/2)^{\binom{v(H_i)}{2}}$. For a tournamenton $W$, we define $t(H,W)$ similarly. We say that a set $S\subseteq \mathbb{R}[\mathcal{T}]$ \emph{forces quasirandomness} if every sequence of tournaments $(T_n)_{n\in \mathbb{N}}$ such that $v(T_n)\to\infty$ and $\lim_{n\to\infty}t(H,T_n)=t(H,1/2)$ for all $H\in S$ is quasirandom. We say that $S$ \emph{forces quasirandomness in regular tournaments} if the same conclusion holds for any sequence of nearly regular tournaments. For $H\in\mathbb{R}[\mathcal{T}]$, we say that $H$ \emph{forces quasirandomness} or \emph{forces quasirandomness in regular tournaments} if the set $\{H\}$ does. The following proposition follows from standard results in combinatorial limit theory; in particular, it is a very slight extension of~\cite[Proposition~1]{Hancock+23}.

\begin{prop}
\label{prop:1/2everywhere}
Let $S\subseteq \mathbb{R}[\mathcal{T}]$. The set $S$ forces quasirandomness if and only if every tournamenton $W$ such that $t(H,W)=t(H,1/2)$ for all $H\in S$ has the property that $W(x,y)=1/2$ for almost all $(x,y)\in [0,1]^2$.
\end{prop}

Given a tournamenton $W$ and $x\in [0,1]$, define the \emph{out-degree} of $x$ to be
\[d_W^+(x):=\int_0^1W(x,y)dy\]
and the \emph{in-degree} to be
\[d_W^-(x):=\int_0^1W(y,x)dy.\]
We say that $W$ is \emph{regular} if $d_W^+(x)=1/2$ for almost every $x\in[0,1]$. Note that, in contrast to the usual in- and out-degrees in tournaments, degrees in tournamentons are always between zero and one. One should think of these quantities as being analogous to the ``normalized'' degrees $d_T^+(v)/v(T)$ and $d_T^-(v)/v(T)$ in a tournament $T$ rather than the degrees themselves.


\section{Basic Reductions}
\label{sec:reductions}


We now obtain a mild extension of a classical result of Kendall and Babington Smith~\cite{KendallBabingtonSmith40} which roughly says that, among tournaments $T$ on $n$ vertices, the homomorphism density of $TT_3$ in $T$ is minimized by a tournament which is ``as regular as possible.'' The following result is well known.

\begin{prop}
\label{prop:regularTransitiveW}
Every tournamenton $W$ satisfies $t(TT_3,W)\geq 1/8$ with equality if and only if $W$ is regular.
\end{prop}

\begin{proof}
Let $W$ be any tournameton. Then
\[t(TT_3,W) = \int_0^1\int_0^1\int_0^1W(x,y)W(x,z)W(y,z)dxdydz\]
but also
\[t(TT_3,W) = \int_0^1\int_0^1\int_0^1W(x,y)W(x,z)W(z,y)dxdydz.\]
Therefore, summing these expressions and using that $W(y,z)+W(z,y)=1$, we get
\begin{equation}
\label{eq:outdeg}
\begin{gathered}
2t(TT_3,W) = \int_0^1\int_0^1\int_0^1W(x,y)W(x,z)dxdydz= \\\int_0^1\left(\int_0^1W(x,y)dy\right)\left(\int_0^1W(x,z)dz\right)dx =\int_0^1 \left(d^+_W(x)\right)^2dx.
\end{gathered}
\end{equation}
Since $W$ is a tournamenton, we have that $\int_0^1 d_W^+(x)dx=1/2$. So, by Jensen's Inequality, 
\[\int_0^1 \left(d^+_W(x)\right)^2dx \geq \left(\int_0^1d_W^+(x)dx\right)^2 = 1/4.\]
Combining this with \eqref{eq:outdeg}, we get $t(TT_3,W)\geq 1/8$. By strict convexity of the function $f(z)=z^2$, we have that equality holds in the above inequality if and only if $d_W^+(x)=d_W^+(y)$ for almost all $x$ and $y$, which can only occur if $d_W^+(x)=1/2$ for almost all $x$.
\end{proof}

We now derive an analogous, but opposite, result for $C_3$; recall from Figure~\ref{fig:smallTourns} that $C_3$ is the cyclic tournament on $3$ vertices. The following result is also well known.

\begin{prop}
\label{prop:regularCyclicW}
Every tournamenton $W$ satisfies $t(C_3,W)\leq 1/8$ with equality if and only if $W$ is regular.
\end{prop}

\begin{proof}
For almost every $x\in [0,1]$, we have $d_W^+(x)+d_W^-(x)=1$. Thus,
\[1=\int_0^1 \left(d_W^+(x)+d_W^-(x)\right)^2dx = \int_0^1 \left(d_W^+(x)\right)^2dx +2\int_0^1d_W^+(x)d_W^-(x) +\int_0^1 \left(d_W^-(x)\right)^2dx.\]
By \eqref{eq:outdeg}, the first term in the rightmost expression above is $2\cdot t(TT_3,W)$. A similar argument to the proof of \eqref{eq:outdeg} implies that the last term is $2\cdot t(TT_3,W)$ as well, and that the middle term is $2t(C_3,W)+2t(TT_3,W)$. 
Therefore, every tournamenton $W$ satisfies
\begin{equation}\label{eq:T3C3}2t(C_3,W) + 6t(TT_3,W)=1\end{equation}
or, in other words,
\[t(C_3,W)=\frac{1}{2}-3t(TT_3,W).\]
The result now follows by Proposition~\ref{prop:regularTransitiveW}.
\end{proof}

\begin{rem}
\label{rem:degreeCount}
By following an argument similar to the proof of Proposition~\ref{prop:regularCyclicW}, but for a finite tournament $T$ rather than a tournamenton $W$, one gets
\[v(T)(v(T)-1)^2=\sum_{v\in V(T)}\left(d_T^+(v)+d_T^-(v)\right)^2 = 2\hom(C_3,T) + 6\hom(TT_3,T) + 2\hom(TT_2,T).\]
\end{rem}

One can obtain the following analogues of the previous propositions for nearly regular sequences of tournaments by following essentially the same proofs.

\begin{prop}
\label{prop:regularTransitiveT}
Every sequence $(T_n)_{n\in \mathbb{N}}$ of tournaments with $v(T_n)\to\infty$ satisfies 
\[\liminf_{n\to\infty}t(TT_3,T_n)\geq 1/8.\]
Moreover, $(T_n)_{n\in \mathbb{N}}$ is nearly regular if and only if $\lim_{n\to\infty}t(TT_3,T_n)= 1/8$.
\end{prop}

\begin{prop}
\label{prop:regularCyclicT}
Every sequence $(T_n)_{n\in \mathbb{N}}$ of tournaments with $v(T_n)\to\infty$ satisfies 
\[\limsup_{n\to\infty}t(C_3,T_n)\leq 1/8.\] 
Moreover, $(T_n)_{n\in \mathbb{N}}$ is nearly regular if and only if $\lim_{n\to\infty}t(C_3,T_n)= 1/8$.
\end{prop}

Using the previous two propositions, we derive the following characterizations of sets of formal linear combinations of tournaments which force quasirandomness in regular tournaments.

\begin{prop}
\label{prop:3vertex}
Let $S\subseteq\mathbb{R}[\mathcal{T}]$. The following are equivalent:
\begin{itemize}
    \item $S$ forces quasirandomness in regular tournaments,
    \item $S\cup\{TT_3\}$ forces quasirandomness,
    \item $S\cup\{C_3\}$ forces quasirandomness.
\end{itemize}
\end{prop}

\begin{proof}
We will only prove that the first statement is equivalent to the second; the analogous result with $TT_3$ replaced by $C_3$ follows from the same arguments except that we apply Proposition~\ref{prop:regularCyclicT} instead of Proposition~\ref{prop:regularTransitiveT}. 

Suppose that $S$ forces quasirandomness in regular tournaments. Let $(T_n)_{n\in \mathbb{N}}$ be a sequence of tournaments with $v(T_n)\to\infty$ and $\lim_{n\to\infty}t(TT_3,T_n)=1/8$. Then $(T_n)_{n\in \mathbb{N}}$ is nearly regular by Proposition~\ref{prop:regularTransitiveT}. So, if we additionally assume that $\lim_{n\to\infty}t(H,T_n)=t(H,1/2)$ for all $H\in S$, then, since $S$ forces quasirandomness in regular tournaments, we get that $(T_n)_{n\in \mathbb{N}}$ is quasirandom. Conversely, assume that $S\cup\{TT_3\}$ forces quasirandomness. Suppose that $(T_n)_{n\in \mathbb{N}}$ is a nearly regular sequence of tournaments such that $\lim_{n\to\infty}t(H,T_n)=t(H,1/2)$ for all $H\in S$. Since the sequence is nearly regular, Proposition~\ref{prop:regularTransitiveT} tells us that $\lim_{n\to\infty}t(TT_3,T_n)=1/8$. Therefore, since $S\cup\{TT_3\}$ forces quasirandomness, we have that $(T_n)_{n\in \mathbb{N}}$ is quasirandom. 
\end{proof}

We now get a sufficient condition for forcing quasirandomness in regular tournaments which we will use in the flag algebra proofs of Section~\ref{sec:flags}. 

\begin{lem}
\label{lem:linear}
Let $H$ be a tournament. If there exists $\alpha,\beta\in \mathbb{R}$ such that the formal linear combination $\alpha TT_3 + \beta H$ forces quasirandomness, then $H$ forces quasirandomness in regular tournaments. Similarly, if there exists $\alpha,\beta\in \mathbb{R}$ such that the formal linear combination $\alpha C_3 + \beta H$ forces quasirandomness, then $H$ forces quasirandomness in regular tournaments.
\end{lem}

\begin{proof}
Suppose that $\alpha TT_3+\beta H$ forces quasirandomness. Then the set $\{TT_3,H\}$ forces quasirandomness, since any sequence $(T_n)_{n\in \mathbb{N}}$ satisfying $\lim_{n\to\infty}t(TT_3,T_n)=1/8$ and $\lim_{n\to\infty}t(H,T_n)=t(H,1/2)$ must satisfy $\lim_{n\to\infty}t(\alpha TT_3+\beta H,T_n) = t(\alpha TT_3+\beta H,1/2)$ as well. So, by Proposition~\ref{prop:3vertex}, $H$ forces quasirandomness in regular tournaments. 
\end{proof}

We close this section with a simple necessary condition for forcing quasirandomness in regular tournaments that will be used in Section~\ref{sec:constructions}. 

\begin{lem}
\label{lem:IVT}
Let $H\in \mathbb{R}[\mathcal{T}]$. If there exist regular tournamentons $W_0$ and $W_1$ such that 
\[t(H,W_0) < t(H,1/2),\text{ and}\]
\[t(H,W_1) > t(H,1/2),\]
then $H$ does not force quasirandomness in regular tournaments. 
\end{lem}

\begin{proof}
For each $z\in (0,1)$, let $W_z$ be the tournamenton defined by
\[W_z(x,y)=\begin{cases}
W_1(x/z,y/z) & \text{if }0\leq x,y\leq z,\\
W_0((x-z)/(1-z), (y-z)/(1-z)) & \text{if }z< x,y\leq 1,\\
1/2 & \text{otherwise}.
\end{cases}\]
An illustration of the structure of this graphon in the case $z=1/4$ is given below. Note that, in all pictures of tournamentons, the origin of $[0,1]^2$ is in the top left corner, in analogy with adjacency matrices. 

\begin{center}
\begin{tikzpicture}
    \draw[thick] (0,0) rectangle (4,4);
    
    \draw[thick] (1,0) -- (1,4);
    
    \draw[thick] (0,3) -- (4,3);
    
    \node at (0.5, 3.5) {\(W_1\)};          
    \node at (2.5, 3.5) {\(\frac{1}{2}\)};  
    \node at (0.5, 1.5) {\(\frac{1}{2}\)};  
    \node at (2.5, 1.5) {\(W_0\)};          

\end{tikzpicture}
\end{center}

Then $t(H,W_z)$ is a continuous function of the variable $z$ for $z\in[0,1]$ and satisfies $t(H,W_0)<t(H,1/2)$ and $t(H,W_1)>t(H,1/2)$. So, by the Intermediate Value Theorem, there must exist $z\in(0,1)$ such that $t(H,W_z)=t(H,1/2)$. It is easily observed that $W_z$ is regular for all $z\in (0,1)$ because $W_0$ and $W_1$ are regular.

Next, we claim that it is not the case that $W_z(x,y)=1/2$ 
for almost all $(x,y)\in[0,z]^2$. Indeed, if it were the case, then, since $z>0$, we would have that $W_1(x,y)=1/2$ for almost all $(x,y)\in[0,1]^2$ as well. However, this would imply $t(H,W_1)=t(H,1/2)$, which contradicts the assumption of the lemma. So, we have $W_z(x,y)\neq 1/2$ for a positive measure of points $(x,y)\in[0,z]^2$. Thus, by Proposition~\ref{prop:1/2everywhere} and Proposition~\ref{prop:regularCyclicW},
the pair $\{H,C_3\}$ does not force quasirandomness which, by Proposition~\ref{prop:3vertex}, implies that $H$ does not force quasirandomness in regular tournaments. 
\end{proof}

\section{Negative Results}
\label{sec:constructions}

In this section, we show that $H_0,H_1,H_2,H_3,H_{12},H_{18}$ and $H_{19}$ do not force quasirandomness in regular tournaments. The following class of tournamentons will be useful for this purpose. Given a finite tournament $T$ with a set of vertex $\{1,\dots,k\}$, one can define a corresponding tournament $W_T$ by dividing $[0,1]$ into intervals $I_1,\dots,I_k$ of equal measure and defining
\[W_T(x,y):=\begin{cases}1 & \text{if }x\in I_i\times I_j\text{ for some }ij\in A(T),\\
1/2 & \text{if }x\in I_i\times I_i\text{ for some }i\in V(T),\\ 0&\text{otherwise}.\end{cases}\]
Say that a tournament $T$ is \emph{regular} if every vertex of $T$ has out-degree $\frac{v(T)-1}{2}$. 

\begin{obs}
Let $T$ be a tournament. Then $T$ is regular if and only if $W_T$ is regular.
\end{obs}
We start with the ``easy cases'' $H_0,H_1,H_2$ and $H_3$.

\begin{prop}
\label{prop:alreadyKnownNeg}
The tournaments $H_0,H_1,H_2$ and $H_3$ do not force quasirandomness in regular tournaments. 
\end{prop}

\begin{proof}
Consider any regular tournamenton $W$ such that there is a subset of $[0,1]^2$ of positive measure on which $W\neq 1/2$. For example, one could take $W=W_T$ for any regular tournament $T$; e.g., $T=C_3$ would work. 
We have $t(H_0,W)=1$ and $t(H_1,W)=1/2$, as these equations hold for every tournamenton. Since $H_2=TT_3$ and $H_3=C_3$, we have $t(H_2,W)=t(H_3,W)=1/8$ by Propositions~\ref{prop:regularTransitiveW} and~\ref{prop:regularCyclicW}. So, by Propositons~\ref{prop:1/2everywhere} and~\ref{prop:3vertex}, none of these tournaments force quasirandomness in regular tournaments. 
\end{proof}

Next, we consider the tournaments $H_9$ and $H_{16}$. 

\begin{prop}
\label{prop:special}
The tournaments $H_9$ and $H_{16}$ do not force quasirandomness in regular tournaments.\footnote{The example in the proof of this proposition was pointed out to us by Bernard Lidick\'y, Daniel Kr\'a\v{l}, Florian Pfender and Jan Volec after posting an earlier draft of this paper to arxiv; see the acknowledgements at the end of the paper.} 
\end{prop}

\begin{proof}
Note that $H_9$ and $H_{16}$ can be obtained from one another by reversing all arcs. Therefore, one of them forces quasirandomness in regular tournaments if and only if the other does. So, without loss of generality, we focus only on $H_9$. Let the five vertices of $H_9$ be denoted by $u_1,\dots,u_5$ taken in clockwise order starting with the top vertex in the picture of $H_9$ in Figure~\ref{fig:smallTourns}. Then $u_1$ is a sink, the only out-neighbour of $u_2$ is $u_1$ and the vertices $u_3,u_4,u_5$ form a cyclic triangle. 

The tournamenton $W_{C_3}$ is regular, and so we will be done if we can argue that $t(H_9,W_{C_3})=2^{-10}$. Let us compute this by hand. We take a $W_{C_3}$-random tournament on five vertices $v_1,\dots,v_5$ and compute the probability that $v_iv_j$ is an arc of this random tournament if and only if $u_iu_j$ is an arc of $H_9$. Let $I_1,I_2$ and $I_3$ be the three intervals used in the construction of $W_{C_3}$. First, if no two of the points $v_3,v_4,v_5$ are sampled from the same such interval, then the probability that $v_1$ is a sink is zero. Next, if $v_3,v_4,v_5$ are sampled from precisely two of the intervals $I_1,I_2,I_3$, then the probability that they form a cyclic triangle is zero. 

Finally, suppose that $v_1,v_2,v_3$ are sampled from the same such interval. There are three choices of interval and, for any choice, the probability of all three of these points landing in that interval is $(1/3)^3$. Given this, the probability that they have the correct arcs among them is $(1/2)^3$. Now, the only ways in which the $W_{C_3}$-random tournament can have the correct arcs with non-zero probability are if (a) $v_1$ and $v_2$ are both mapped to the same interval as $v_3,v_4,v_5$, (b) $v_2$ is mapped to the same interval as $v_3,v_4,v_5$ and $v_1$ is mapped to the next interval in cyclic order, or (c) both $v_1$ and $v_2$ are mapped to the next interval in cyclic order. The probabilities of each of these events, and the probability that the $W_{C_3}$-random tournament has the correct arcs given these events, can be computed fairly straightforwardly to give the following:
\[t(H_9,W_{C_3}) = 3\left(\frac{1}{3}\right)^3\left(\frac{1}{2}\right)^3\left(\frac{1}{3}\right)^2\left(\left(\frac{1}{2}\right)^7 + \left(\frac{1}{2}\right)^3+\left(\frac{1}{2}\right)^1\right)=2^{-10}\]
as desired. 
\end{proof}

We conclude the section by using Lemma~\ref{lem:IVT} to prove that $H_{12},H_{18}$ and $H_{19}$ do not force quasirandomness in regular tournaments.

\begin{lem}
\label{lem:negFive}
The tournaments $H_{12},H_{18}$ and $H_{19}$ do not force quasirandomness in regular tournaments. 
\end{lem}

\begin{proof}
We will prove the lemma via three applications of Lemma~\ref{lem:IVT}. For each $z\in[0,1/2]$, let $U_z$ be the tournamenton defined by setting
\[U_z(x,y)=1/2 +z(2\cdot W_{C_3}(x,y) - 1)\]
for all $x,y\in[0,1]$. Visually, $U_z$ is the function obtained by dividing $[0,1]^2$ into a regular $3\times 3$ grid and assigning values to points in each square of the grid according to the following diagram.
\begin{center}
\begin{tikzpicture}[scale=1.5]
    \draw (0,0) grid (3,3);
    \node at (0.5, 0.5) {\(\frac{1}{2}+z\)};
    \node at (1.5, 0.5) {\(\frac{1}{2} - z\)};
    \node at (2.5, 0.5) {\(\frac{1}{2}\)};
    
    \node at (0.5, 1.5) {\(\frac{1}{2} - z\)};
    \node at (1.5, 1.5) {\(\frac{1}{2} \)};
    \node at (2.5, 1.5) {\(\frac{1}{2} + z\)};
    
    \node at (0.5, 2.5) {\(\frac{1}{2}\)};
    \node at (1.5, 2.5) {\(\frac{1}{2} + z\)};
    \node at (2.5, 2.5) {\(\frac{1}{2} - z\)};
\end{tikzpicture}
\end{center}

Note that $U_0$ is nothing more than the constant function equal to $1/2$ everywhere, whereas $U_{1/2}$ is $W_{C_3}$. Also note that $U_z$ is a regular tournamenton for all $z\in[0,1]$. It is tedious, but not hard, to show that
\[t(H_{12},U_z)=\frac{1}{1024}- \frac{1}{96}z^4+\frac{7}{108}z^6+ \frac{1}{18}z^8,\]
\[t(H_{18},U_z)=\frac{1}{1024}- \frac{1}{96}z^4-\frac{1}{108}z^6- \frac{7}{162}z^8,\text{ and}\]
\[t(H_{19},U_z) = \frac{1}{1024}+ \frac{5}{288}z^4-\frac{5}{36}z^6+ \frac{5}{162}z^8.\]
From this, we see that
\[t(H_{12},U_{1/2})=\frac{43}{27648}>\frac{1}{1024},\]
\[t(H_{12},U_{1/3})= \frac{57161}{60466176}<\frac{1}{1024},\]
\[t(H_{18},U_{1/2})= \frac{1}{82944}<\frac{1}{1024},\]
\[t(H_{19},U_{1/3})= \frac{546961}{544195584}>\frac{1}{1024},\]
\[t(H_{19},U_{1/2})= \frac{1}{82944}<\frac{1}{1024}.\]
By Lemma~\ref{lem:IVT}, we get that $H_{12}$ and $H_{19}$ do not force quasirandomness in regular tournaments. The last thing that we need is a tournamenton $W$ such that $t(H_{18},W)>\frac{1}{1024}$. Let $T$ be the $7$-vertex tournament with adjacency matrix 
\[\begin{bmatrix}
0 & 1 & 1 & 1 & 0 & 0 & 0\\
0 & 0 & 1 & 0 & 1 & 1 & 0\\
0 & 0 & 0 & 1 & 1 & 0 & 1\\
0 & 1 & 0 & 0 & 0 & 1 & 1\\
1 & 0 & 0 & 1 & 0 & 1 & 0\\
1 & 0 & 1 & 0 & 0 & 0 & 1\\
1 & 1 & 0 & 0 & 1 & 0 & 0
\end{bmatrix}.\]
Clearly, $T$ is regular. This tournament was found by an exhaustive computer search through all tournaments on a small number of vertices. By computer, we have found that  
\[t(H_{18},W_T) \approx 0.001249886 > \frac{1}{1024}.\]
Thus, we get that $H_{18}$ does not force quasirandomness in regular tournaments by one more application of Lemma~\ref{lem:IVT}. 
\end{proof}

\section{Positive Results}
\label{sec:flags}

In this section, we complete the proof of Theorem~\ref{th:main} by showing that all tournaments in part \ref{thm:doForce} of the theorem force quasirandomness in regular tournaments. We start with five examples that can be derived straightforwardly from known results. 

\begin{prop}[See~\cite{Lovasz79,CoreglianoParenteSato19}]
\label{prop:alreadyKnownPos}
The tournaments $H_4,H_5,H_7,H_8$ and $H_{17}$ force quasirandomness in regular tournaments. 
\end{prop}

\begin{proof}
The tournaments $H_4$, $H_8$ and $H_{17}$ force quasirandomness by Theorem~\ref{th:qriff} (note that $H_4$ and $H_8$ are $TT_4$ and $TT_5$, respectively), and so they force quasirandomness in regular tournaments as well. 

We consider $H_5$ and let $(T_n)_{n\in \mathbb{N}}$ be a nearly regular sequence of tournaments. Note that $H_5$ consists of a cycle on three vertices and one ``sink'' vertex; i.e. a vertex of out-degree zero. For each $v\in V(T_n)$, let $N_n^-(v)$ denote the in-neighbourhood of $v$ and $d_n^-(v)$ the in-degree of $v$. For a set $S\subseteq V(T_n)$, let $T_n[S]$ be the subtournament of $T_n$ induced by $S$. Then
\[\hom(H_5,T_n) = \sum_{v\in V(T_n)}\hom(C_3,T[N_n^-(v)]).\]
To see this, we note that the term corresponding to $v$ counts homomorphisms from $H_5$ to $T_n$ in which the sink is mapped to $v$. Since $\hom(C_3,T[N_n^-(v)])=d_n^-(v)^3\cdot t(C_3,T[N_n^-(v)])$ by definition, this can be rewritten as
\[\hom(H_5,T_n) = \sum_{v\in V(T_n)} \left(d_n^-(v)\right)^3t(C_3,T_n[N_n^-(v)]).\]
Using Remark~\ref{rem:degreeCount}, we can rewrite this again as
\[\hom(H_5,T_n) = \sum_{v\in V(T_n)} \left(d_n^-(v)\right)^3\frac{1}{2}\left(1 -6t(TT_3,T_n[N_n^-(v)]) -o(1)\right).\]
Now, if we use the fact that $(T_n)_{n\in \mathbb{N}}$ is nearly regular, the right side of the above equality becomes
\[(v(T_n)/2)^4 - 3\sum_{v\in V(T_n)} \left(d_n^-(v)\right)^3t(TT_3,T_n[N_n^-(v)]) - o(v(T_n)^4)\]
\[=(v(T_n)/2)^4 - 3\hom(TT_4,T_n) - o(v(T_n)^4)\]
where the fact that $\sum_{v\in V(T_n)} \left(d_n^-(v)\right)^3t(TT_3,T_n[N_n^-(v)]) = \hom(TT_4,T_n)$ follows from a similar ``double-counting'' argument as we used before to express $\hom(H_5,T_n)$ as a summation. 
If we divide this by $v(T_n)^4$, we get
\[t(H_5,T_n) = (1/2)^4 - 3t(TT_4,T_n)-o(1).\]
So, if $\lim_{n\to\infty}t(H_5,T_n)=(1/2)^6$, then $\lim_{n\to\infty}t(TT_4,T_n)= (1/2)^6$. Since $TT_4$ forces quasirandomness (see Theorem~\ref{th:qriff}), the sequence $(T_n)_{n\in \mathbb{N}}$ must be quasirandom, and hence
we get that $H_5$ forces quasirandomness in regular tournaments. Note that $H_7$ is a cycle on three vertices with a source vertex (i.e. a vertex of in-degree zero) and so the fact that it forces quasirandomness in regular tournaments follows from a very similar argument.
\end{proof}

We next consider $C_4=H_6$. The fact that $H_6$ forces quasirandomness in regular tournaments is easily derived from the following result of~\cite{ChanGrzesikKralNoel20}.\footnote{Note that, in~\cite{ChanGrzesikKralNoel20}, they use $C_4$ to denote the $4$-vertex directed cycle digraph, whereas here we let $C_4$ denote the unique $4$-vertex tournament containing such a cycle. Because of this, the value of $t(C_4,W)$ in this paper is precisely $1/4$ times the value of $t(C_4,W)$ in~\cite{ChanGrzesikKralNoel20} for every tournamenton $W$.} 

\begin{lem}[Chan, Grzesik, Kr\'a\v{l} and Noel~{\cite[Corollary~6]{ChanGrzesikKralNoel20}}]
\label{lem:C4}
If $W$ is a tournamenton and $1/2\leq z\leq 1$ such that 
\[t(C_3,W) =\frac{1}{8}\left(z^3 + (1-z)^3\right),\]
then
\[t(C_4,W)\geq \frac{1}{64}\left(z^4 + (1-z)^4\right)\]
where equality holds if and only if there exists a measurable function $f:[0,1] \to [0,1/2]$ such that $W(x,y) = 1/2 + f(x) - f(y)$ for almost every $(x, y) \in [0,1]^2$. 
\end{lem}

\begin{prop}[See~{\cite[Corollary~6]{ChanGrzesikKralNoel20}}]
\label{prop:C4}
$H_6$ forces quasirandomness in regular tournaments. 
\end{prop}

\begin{proof}
Note that $H_6$ is $C_4$. Let $W$ be a regular tournamenton. By Proposition~\ref{prop:regularCyclicW}, we have $t(C_3,W)=1/8$. So, applying Lemma~\ref{lem:C4} with $z=1$ gives us that $t(C_4,W)\geq 1/64$ where equality holds if and only if there is a measurable function $f:[0,1]\to[0,1/2]$ such that $W(x,y)=1/2+f(x)-f(y)$ for all $x,y\in [0,1]$. Since $W$ is regular, we have, for almost all $x\in[0,1]$,
\[\frac{1}{2}=d_W^+(x)=\int_0^1W(x,y)dy = \int_0^1\left(\frac{1}{2}+f(x)-f(y)\right)dy\]
which implies that $f(x) = \int_0^1f(y)dy$. Therefore, the function $f$ is constant almost everywhere which tells us that $W(x,y)=1/2$ for almost all $(x,y)\in [0,1]^2$. So, we are done by Proposition~\ref{prop:1/2everywhere}.
\end{proof}

The goal in the rest of this section is to prove the following five theorems which, together with Lemma~\ref{lem:linear} 
and Proposition~\ref{prop:1/2everywhere}, imply that $H_{10}, H_{11},H_{13}$ and $H_{14}$ force quasirandomness in regular tournaments. Note that $H_{15}$ is the tournament obtained from $H_{10}$ by reversing all the arcs. Therefore, we will get that $H_{15}$ forces quasirandomness in regular tournaments as well.


\begin{thm}
\label{th:H10}
For every tournamenton $W$,
\[8\cdot t(C_3,W) + \frac{1}{4}\cdot1024\cdot t(H_{10},W) \leq \frac{5}{4}\]
where equality holds if and only if $W(x,y)=1/2$ for almost all $(x,y)\in[0,1]^2$.
\end{thm}

\begin{thm}
\label{th:H11}
For every tournamenton $W$,
\[8\cdot t(C_3,W) + \frac{1}{4}\cdot1024\cdot t(H_{11},W) \leq \frac{5}{4}\]
where equality holds if and only if $W(x,y)=1/2$ for almost all $(x,y)\in[0,1]^2$.
\end{thm}

\begin{thm}
\label{th:H13}
For every tournamenton $W$,
\[8\cdot t(TT_3,W) + \frac{1}{7}\cdot1024\cdot t(H_{13},W) \geq \frac{8}{7}\]
where equality holds if and only if $W(x,y)=1/2$ for almost all $(x,y)\in[0,1]^2$.
\end{thm}

\begin{thm}
\label{th:H14}
For every tournamenton $W$,
\[8\cdot t(TT_3,W) + \frac{1}{5}\cdot1024\cdot t(H_{14},W) \geq \frac{6}{5}\]
where equality holds if and only if $W(x,y)=1/2$ for almost all $(x,y)\in[0,1]^2$.
\end{thm}

We now deduce our main result from these five theorems, after which we will prove the theorems themselves.

\begin{proof}[Proof of Theorem~\ref{th:main}]
By Propositions~\ref{prop:1/2everywhere},~\ref{prop:alreadyKnownPos} and~\ref{prop:C4} and Theorems~\ref{th:H10},~\ref{th:H11},~\ref{th:H13} and~\ref{th:H14}, together with the fact that and $H_{15}$ is obtained from  $H_{10}$ by reversing all arcs, each of the tournaments $H_i$ for $i\in\{4,5,6,7,8,10,11,13,14,15,17\}$ forces quasirandomness in regular tournaments.  By Proposition~\ref{prop:alreadyKnownNeg} and Lemma~\ref{lem:negFive}, the tournaments $H_i$ for $i\in\{0,1,2,3,12,18,19\}$ do not. This completes the proof.
\end{proof}

We now focus our attention on the proofs of Theorems~\ref{th:H10}--\ref{th:H14}. We will apply the \emph{flag algebra} method, introduced by Razborov~\cite{Razborov07}. This method has had a heavy impact throughout extremal combinatorics, including several successful applications to problems involving tournaments~\cite{CoreglianoParenteSato19,CoreglianoRazborov17,BurkeLidickyPfenderPhilips22+,LinialMorgenstern16}. Our presentation of the method will follow along similar lines to the treatment in~\cite[Section~5]{MossNoel23+} for graphs. 

For each tournament $J$ and tournamenton $W$, the \emph{induced density} of $J$ in $W$ is defined to be 
\begin{equation}\label{eq:dt}d(J,W) := \frac{v(J)!}{\aut(J)}\cdot t(J,W)\end{equation} 
where $\Aut(J)$ is the set of all automorphisms of $J$ and $\aut(J):=|\Aut(J)|$. It is not hard to see that $d(J,W)$ is precisely the probability that a $W$-random tournament with $v(J)$ vertices is isomorphic to $J$. The following lemma just says that these probabilities sum to $1$.

\begin{lem}
\label{lem:sumToOne}
For any $m\geq1$ and tournamenton $W$,
\[\sum_{J: v(J)=m}d(J,W)=1\]
where the sum is over all $m$-vertex tournaments $J$ up to isomorphism. 
\end{lem}

In this paper, a \emph{flag} is a pair $(F,r)$ where $F$ is a tournament on vertex set $\{1,\dots,v(F)\}$ and $r\leq v(F)$. The vertices $\{1,\dots,r\}$ are known as the \emph{roots} of the flag. Given a flag $(F,r)$ and a tournamenton $W$, define $t_r(F,W):[0,1]^r\to [0,1]$ by
\[t_r(F,W)(x_1,\dots,x_r):=\int_0^1\cdots \int_0^1\prod_{ij\in A(F)}W(x_i,x_j)dx_{r+1}\cdots dx_{v(F)}.\]
That is, it is the same as the usual homomorphism density, except that we only integrate over the variables $x_i$ for $i\in\{r+1,\dots,v(F)\}$. Now, for $r\leq k$ and any tournament $R$ with $V(R)=\{1,\dots,r\}$, let $\mathcal{F}_k(R)$ be the collection of all distinct (but possibly isomorphic) flags $(F,r)$ such that $F$ has $k$ vertices and the subtournament of $F$ induced by $\{1,\dots,r\}$ is equal to $R$ (not just isomorphic to $R$, but has exactly the same arc set). Note that $\mathcal{F}_k(R)$ has cardinality $2^{r(k-r)+\binom{k-r}{2}}$. The following lemma uses the same idea as the proof of Proposition~\ref{prop:regularTransitiveW}; that is, it exploits the fact that $W(x,y)+W(y,x)=1$ for every tournamenton $W$ and $(x,y)\in[0,1]^2$. 

\begin{lem}
\label{lem:sumAllOrientations}
Let $R$ be a tournament with $V(R)=\{1,\dots,r\}$ and let $k\geq r$. For any tournamenton $W$ we have
\[\sum_{(F,r)\in \mathcal{F}_k(R)} t_r(F,W)(x_1,\dots,x_r) = t_r(R,W)(x_1,\dots,x_r)\]
for almost every $(x_1,\dots,x_r)\in [0,1]^r$. 
\end{lem}

\begin{proof}
For this proof, it is convenient to extend the notion of a flag slightly to include any pair $(D,r)$ where $D$ is a digraph on vertex set $\{1,\dots,v(D)\}$ and $r\leq v(D)$. The function $t_r(D,W)(x_1,\dots,x_r)$ is defined in the same way as it is for tournaments. Given a tournament $R$ on vertex set $\{1,\dots,r\}$ and a graph $G$ with vertex set $\{1,\dots,m\}$ containing a clique on $\{1,\dots,r\}$, let $\mathcal{F}_G(R)$ be the set of all flags $(D,r)$ where $D$ is obtained from orienting the edges of $G$ so that the subdigraph induced by $\{1,\dots,r\}$ is equal to $R$. We claim that 
\[\sum_{(D,r)\in \mathcal{F}_G(R)} t_r(D,W)(x_1,\dots,x_r) = t_r(R,W)(x_1,\dots,x_r)\]
for almost every $(x_1,\dots,x_r)\in [0,1]^r$. We proceed by induction on $e(G)$. First, if every edge of $G$ is in $\{1,\dots,r\}$, then the equality is clear, since integrating out the variables corresponding to vertices not in $\{1,\dots,r\}$ just produces factors of $1$. Now, let $uv$ be any edge of $G$ not contained in $\{1,\dots,r\}$. For any $(D,r)\in\mathcal{F}_G(R)$ containing an arc from $u$ to $v$, there is a flag $(D',r)\in \mathcal{F}_G(R)$ with all the same arcs as $D$ except that the edge $uv$ is oriented in the other direction. Let $D''$ be obtained from $D$ by deleting the arc $uv$. Since $W(x,y)+W(y,x)=1$ for almost all $(x,y)\in[0,1]^2$ we have
\[t_r(D,W)(x_1,\dots,x_r) +  t_r(D',W)(x_1,\dots,x_r) = t_r(D'',W)(x_1,\dots,x_r).\]
Pairing up all digraphs in $\mathcal{F}_G(R)$ in this way gives us that
\[\sum_{(D,r)\in \mathcal{F}_G(R)} t_r(D,W)(x_1,\dots,x_r) = \sum_{(D'',r)\in \mathcal{F}_{G\setminus \{e\}}(R)} t_r(D'',W)(x_1,\dots,x_r)\]
and we are done by induction.
\end{proof}

Given tournaments $H$ and $J$ with $v(H)\leq v(J)$, define the \emph{injective homomorphism density} $t_{\inj}(H,J)$ to be the probability that a uniformly random
injective function from $V(H)$ to $V(J)$ is a homomorphism. That is, it is the number of injective homomorphisms from $H$ to $J$ divided by the number of injective functions from $V(H)$ to $V(J)$.
In the case that $v(H)>v(J)$, simply define $t_{\inj}(H,J)$ to be zero. Let $\mathcal{T}_{\leq m}$ be the set of all tournaments on at most $m$ vertices up to isomorphism and $\mathbb{R}[\mathcal{T}_{\leq m}]$ be the set of all formal linear combinations of elements of $\mathcal{T}_{\leq m}$. We extend $t_{\inj}(\cdot,W)$ to elements of $\mathbb{R}[\mathcal{T}_{\leq m}]$ linearly, analogous to $t(\cdot,W)$. We derive the next lemma from Lemma~\ref{lem:sumAllOrientations}. 


\begin{lem}
\label{lem:t to tinj}
For any $m\geq 1$, $H\in\mathbb{R}[\mathcal{T}_{\leq m}]$ and tournamenton $W$,
\[t(H,W) = \sum_{J: v(J)=m}t_{\inj}(H,J)\cdot d(J,W)\]
where the sum is over all $m$-vertex tournaments $J$ up to isomorphism. 
\end{lem}

\begin{proof}
We may assume that $H$ is a single graph, as the general result will follow from linearity. Let $r=v(H)$ and note that we may assume $V(H)=\{1,\dots,r\}$. By Lemma~\ref{lem:sumAllOrientations}, we have that 
\begin{equation}\label{eq:t to tinj}t_r(H,W)(x_1,\dots,x_r) = \sum_{(F,r)\in \mathcal{F}_m(H)} t_r(F,W)(x_1,\dots,x_r)\end{equation}
for almost every $(x_1,\dots,x_r)\in [0,1]^r$. Integrating this over all $(x_1,\dots,x_r)\in [0,1]^r$ gives us $t(H,W)$ on the left side. 

On the right side, we get $2^{r(m-r)+\binom{m-r}{2}}$ terms, each of which is of the form $t(J,W)$ for some graph $J$ on $m$ vertices. For each graph $J$ on $m$ vertices, up to isomorphism, let us compute the coefficient of $t(J,W)$. If $\hom(H,J)=0$, then the coefficient is zero, so we assume that there is at least one injective homomorphism from $H$ to $J$. So, we may assume that $V(J)=\{1,\dots,m\}$ and the subtournament of $J$ on $\{1,\dots,r\}$ is equal to $H$. Consider the subgroup $\Gamma$ of the symmetric group $S_m$ consisting of bijections on $\{1,\dots,m\}$ whose restriction to $\{1,\dots,r\}$ is a homomorphism from $H$ to $J$. Then $\Aut(J)$ is a subgroup of $\Gamma$. Now, the coefficient of $t(J,W)$ on the right side of \eqref{eq:t to tinj} is equal to the number of cosets of $\Aut(J)$ in $\Gamma$ which, by standard facts from group theory, is $\frac{\hom_{\inj}(H,J)(m-r)!}{\aut(J)}$. Putting all of this together,
\[t(H,W) = \sum_{J:v(J)=m}\frac{\hom_{\inj}(H,J)(m-r)!}{\aut(J)}t(J,W)\]
\[= \sum_{J:v(J)=m}\left(\frac{\hom_{\inj}(H,J)(m-r)!}{m!}\right)\cdot \left(\frac{m!}{\aut(J)}t(J,W)\right)=\sum_{J:v(J)=m}t_{\inj}(H,J)d(J,W).\]
\end{proof}

At this point, let us briefly pause our discussion of the flag algebra method to prove the following lemma which will be very useful in characterizing the equality cases in Theorems~\ref{th:H10}--\ref{th:H14}. 

\begin{lem}
\label{lem:uniqueness}
If $W$ is a tournamenton such that $8t(C_4,W)=t(TT_3,W)$ and $8t(TT_4,W)=t(TT_3,W)$, then $W=1/2$ almost everywhere. 
\end{lem}

\begin{proof}
There are four tournaments on four vertices and, among these, $TT_4$ contains four transitive triangles, $C_4$ contains two, and the other two tournaments, $H_5$ and $H_7$, contain three each. There are $4!=24$ injective maps from a set of cardinality three to a set of cardinality four. So, applying Lemma~\ref{lem:t to tinj} to the case $H=TT_3$ and $m=4$ yields
\[t(TT_3,W) = \frac{4}{24}d(TT_4,W) + \frac{2}{24}d(C_4,W)+\frac{3}{24}d(H_5,W)+\frac{3}{24}d(H_7,W).\]
Now, using Lemma~\ref{lem:sumToOne} to cancel out the terms involving $H_5$ and $H_7$ gives us
\[t(TT_3,W)-\frac{3}{24} = \frac{1}{24}d(TT_4,W) - \frac{1}{24}d(C_4,W).\]
By \eqref{eq:dt} and the fact that $\aut(TT_4)=\aut(C_4)=1$, this is equivalent to
\[t(TT_3,W)-\frac{1}{8} = t(TT_4,W) - t(C_4,W).\]
Now, using the hypothesis of the lemma, we get that the right side is zero. Therefore, we have $t(TT_3,W)=1/8$. Applying the hypothesis once again gives us $t(TT_4,W)=1/64$ which, by the fact that $TT_4$ forces quasirandomness (see Theorem~\ref{th:qriff}), implies that $W=1/2$ almost everywhere. 
\end{proof}

We now present the flags that we will use in the proofs of Theorems~\ref{th:H10}--\ref{th:H14}. In each depiction of a flag $(F,r)$, the root vertices $1,\dots,r$ are depicted as square nodes, listed left to right in increasing order. Each flag will have exactly one non-root vertex which is depicted as a circular node. We use the flags in $\mathcal{F}_3(TT_2)$, labelled by $F^1_i$ for $1\leq i\leq 4$ as follows:

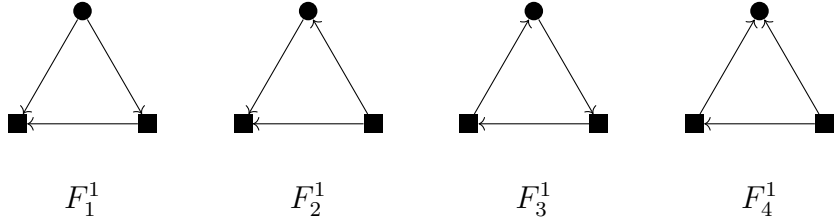
\begin{center}
    \begin{tikzpicture}[scale=1]
        \begin{scope}[xshift=0cm, yshift=0.0cm]
          \node[fill=black, rectangle, inner sep=0pt,minimum size=7pt] at (210:1cm) (1) {};
          \node[fill=black, rectangle, inner sep=0pt,minimum size=7pt] at (330:1cm) (2) {};
          \node[fill=black, circle, inner sep=0pt,minimum size=7pt] at (90:1cm) (3) {};
          \node at (270:0.7cm) (4) {};
          \draw[->] (2) to (1);
          \draw[->] (3) to (1);
          \draw[->] (3) to (2);
          \node (label) at (270:1.5cm) {$F^1_1$};
        \end{scope}
        \begin{scope}[xshift=3cm, yshift=0.0cm]
          \node[fill=black, rectangle, inner sep=0pt,minimum size=7pt] at (210:1cm) (1) {};
          \node[fill=black, rectangle, inner sep=0pt,minimum size=7pt] at (330:1cm) (2) {};
          \node[fill=black, circle, inner sep=0pt,minimum size=7pt] at (90:1cm) (3) {};
          \node at (270:0.7cm) (4) {};
          \draw[->] (2) to (1);
          \draw[->] (3) to (1);
          \draw[->] (2) to (3);
          \node (label) at (270:1.5cm) {$F^1_2$};
        \end{scope}
        \begin{scope}[xshift=6cm, yshift=0.0cm]
          \node[fill=black, rectangle, inner sep=0pt,minimum size=7pt] at (210:1cm) (1) {};
          \node[fill=black, rectangle, inner sep=0pt,minimum size=7pt] at (330:1cm) (2) {};
          \node[fill=black, circle, inner sep=0pt,minimum size=7pt] at (90:1cm) (3) {};
          \node at (270:0.7cm) (4) {};
          \draw[->] (2) to (1);
          \draw[->] (1) to (3);
          \draw[->] (3) to (2);
          \node (label) at (270:1.5cm) {$F^1_3$};
        \end{scope}
        \begin{scope}[xshift=9cm, yshift=0.0cm]
          \node[fill=black, rectangle, inner sep=0pt,minimum size=7pt] at (210:1cm) (1) {};
          \node[fill=black, rectangle, inner sep=0pt,minimum size=7pt] at (330:1cm) (2) {};
          \node[fill=black, circle, inner sep=0pt,minimum size=7pt] at (90:1cm) (3) {};
          \node at (270:0.7cm) (4) {};
          \draw[->] (2) to (1);
          \draw[->] (1) to (3);
          \draw[->] (2) to (3);
          \node (label) at (270:1.5cm) {$F^1_4$};
        \end{scope}
    \end{tikzpicture}
    \end{center}
We also use the flags in $\mathcal{F}_4(TT_3)$, which we label by $F^2_i$ for $1\leq i\leq 8$ as follows:
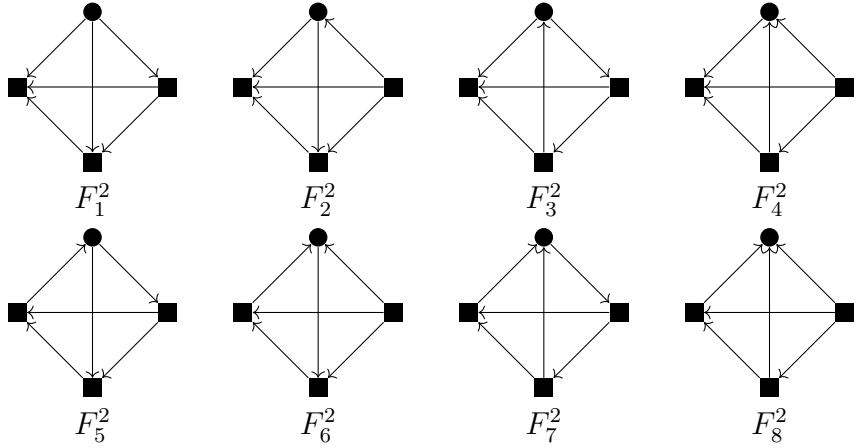
\begin{center}
    \begin{tikzpicture}[scale=1]
        \begin{scope}[xshift=0cm, yshift=0.0cm]
          \node[fill=black, rectangle, inner sep=0pt,minimum size=7pt] at (180:1cm) (1) {};
          \node[fill=black, rectangle, inner sep=0pt,minimum size=7pt] at (270:1cm) (2) {};
          \node[fill=black, rectangle, inner sep=0pt,minimum size=7pt] at (0:1cm) (3) {};
          \node[fill=black, circle, inner sep=0pt,minimum size=7pt] at (90:1cm) (4) {};
          \draw[->] (2) to (1);
          \draw[->] (3) to (1);
          \draw[->] (4) to (1);
          \draw[->] (3) to (2);
          \draw[->] (4) to (2);
          \draw[->] (4) to (3);
          \node (label) at (270:1.5cm) {$F^2_1$};
        \end{scope}
        \begin{scope}[xshift=3cm, yshift=0.0cm]
          \node[fill=black, rectangle, inner sep=0pt,minimum size=7pt] at (180:1cm) (1) {};
          \node[fill=black, rectangle, inner sep=0pt,minimum size=7pt] at (270:1cm) (2) {};
          \node[fill=black, rectangle, inner sep=0pt,minimum size=7pt] at (0:1cm) (3) {};
          \node[fill=black, circle, inner sep=0pt,minimum size=7pt] at (90:1cm) (4) {};
          \draw[->] (2) to (1);
          \draw[->] (3) to (1);
          \draw[->] (4) to (1);
          \draw[->] (3) to (2);
          \draw[->] (4) to (2);
          \draw[->] (3) to (4);
          \node (label) at (270:1.5cm) {$F^2_2$};
        \end{scope}
        \begin{scope}[xshift=6cm, yshift=0.0cm]
          \node[fill=black, rectangle, inner sep=0pt,minimum size=7pt] at (180:1cm) (1) {};
          \node[fill=black, rectangle, inner sep=0pt,minimum size=7pt] at (270:1cm) (2) {};
          \node[fill=black, rectangle, inner sep=0pt,minimum size=7pt] at (0:1cm) (3) {};
          \node[fill=black, circle, inner sep=0pt,minimum size=7pt] at (90:1cm) (4) {};
          \draw[->] (2) to (1);
          \draw[->] (3) to (1);
          \draw[->] (4) to (1);
          \draw[->] (3) to (2);
          \draw[->] (2) to (4);
          \draw[->] (4) to (3);
          \node (label) at (270:1.5cm) {$F^2_3$};
        \end{scope}
        \begin{scope}[xshift=9cm, yshift=0.0cm]
          \node[fill=black, rectangle, inner sep=0pt,minimum size=7pt] at (180:1cm) (1) {};
          \node[fill=black, rectangle, inner sep=0pt,minimum size=7pt] at (270:1cm) (2) {};
          \node[fill=black, rectangle, inner sep=0pt,minimum size=7pt] at (0:1cm) (3) {};
          \node[fill=black, circle, inner sep=0pt,minimum size=7pt] at (90:1cm) (4) {};
          \draw[->] (2) to (1);
          \draw[->] (3) to (1);
          \draw[->] (4) to (1);
          \draw[->] (3) to (2);
          \draw[->] (2) to (4);
          \draw[->] (3) to (4);
          \node (label) at (270:1.5cm) {$F^2_4$};
        \end{scope}
        \begin{scope}[xshift=0cm, yshift=-3.0cm]
          \node[fill=black, rectangle, inner sep=0pt,minimum size=7pt] at (180:1cm) (1) {};
          \node[fill=black, rectangle, inner sep=0pt,minimum size=7pt] at (270:1cm) (2) {};
          \node[fill=black, rectangle, inner sep=0pt,minimum size=7pt] at (0:1cm) (3) {};
          \node[fill=black, circle, inner sep=0pt,minimum size=7pt] at (90:1cm) (4) {};
          \draw[->] (2) to (1);
          \draw[->] (3) to (1);
          \draw[->] (1) to (4);
          \draw[->] (3) to (2);
          \draw[->] (4) to (2);
          \draw[->] (4) to (3);
          \node (label) at (270:1.5cm) {$F^2_5$};
        \end{scope}
        \begin{scope}[xshift=3cm, yshift=-3.0cm]
          \node[fill=black, rectangle, inner sep=0pt,minimum size=7pt] at (180:1cm) (1) {};
          \node[fill=black, rectangle, inner sep=0pt,minimum size=7pt] at (270:1cm) (2) {};
          \node[fill=black, rectangle, inner sep=0pt,minimum size=7pt] at (0:1cm) (3) {};
          \node[fill=black, circle, inner sep=0pt,minimum size=7pt] at (90:1cm) (4) {};
          \draw[->] (2) to (1);
          \draw[->] (3) to (1);
          \draw[->] (1) to (4);
          \draw[->] (3) to (2);
          \draw[->] (4) to (2);
          \draw[->] (3) to (4);
          \node (label) at (270:1.5cm) {$F^2_6$};
        \end{scope}
        \begin{scope}[xshift=6cm, yshift=-3.0cm]
          \node[fill=black, rectangle, inner sep=0pt,minimum size=7pt] at (180:1cm) (1) {};
          \node[fill=black, rectangle, inner sep=0pt,minimum size=7pt] at (270:1cm) (2) {};
          \node[fill=black, rectangle, inner sep=0pt,minimum size=7pt] at (0:1cm) (3) {};
          \node[fill=black, circle, inner sep=0pt,minimum size=7pt] at (90:1cm) (4) {};
          \draw[->] (2) to (1);
          \draw[->] (3) to (1);
          \draw[->] (1) to (4);
          \draw[->] (3) to (2);
          \draw[->] (2) to (4);
          \draw[->] (4) to (3);
          \node (label) at (270:1.5cm) {$F^2_7$};
        \end{scope}
        \begin{scope}[xshift=9cm, yshift=-3.0cm]
          \node[fill=black, rectangle, inner sep=0pt,minimum size=7pt] at (180:1cm) (1) {};
          \node[fill=black, rectangle, inner sep=0pt,minimum size=7pt] at (270:1cm) (2) {};
          \node[fill=black, rectangle, inner sep=0pt,minimum size=7pt] at (0:1cm) (3) {};
          \node[fill=black, circle, inner sep=0pt,minimum size=7pt] at (90:1cm) (4) {};
          \draw[->] (2) to (1);
          \draw[->] (3) to (1);
          \draw[->] (1) to (4);
          \draw[->] (3) to (2);
          \draw[->] (2) to (4);
          \draw[->] (3) to (4);
          \node (label) at (270:1.5cm) {$F^2_8$};
        \end{scope}
    \end{tikzpicture}
    \end{center}
Finally, we have the elements of $\mathcal{F}_4(C_3)$, labelled $F^3_i$ for $1\leq i\leq 8$ like so:
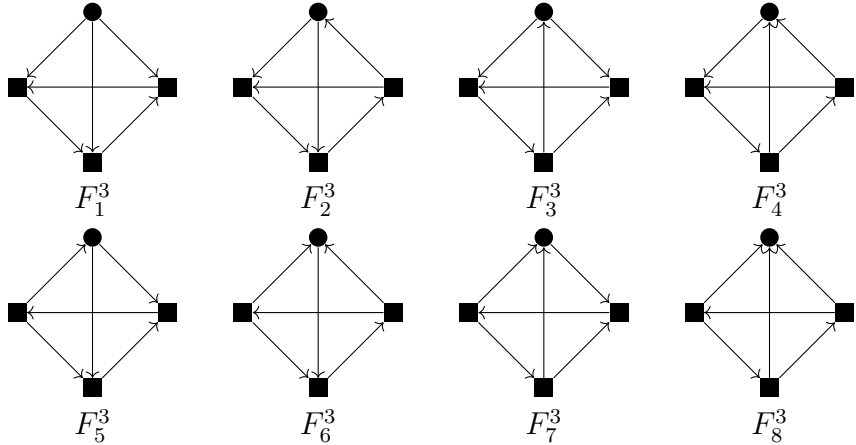
\begin{center}
    \begin{tikzpicture}[scale=1]
        \begin{scope}[xshift=0cm, yshift=0.0cm]
          \node[fill=black, rectangle, inner sep=0pt,minimum size=7pt] at (180:1cm) (1) {};
          \node[fill=black, rectangle, inner sep=0pt,minimum size=7pt] at (270:1cm) (2) {};
          \node[fill=black, rectangle, inner sep=0pt,minimum size=7pt] at (0:1cm) (3) {};
          \node[fill=black, circle, inner sep=0pt,minimum size=7pt] at (90:1cm) (4) {};
          \draw[->] (1) to (2);
          \draw[->] (2) to (3);
          \draw[->] (3) to (1);
          \draw[->] (4) to (1);
          \draw[->] (4) to (2);
          \draw[->] (4) to (3);
          \node (label) at (270:1.5cm) {$F^3_1$};
        \end{scope}
        \begin{scope}[xshift=3cm, yshift=0.0cm]
          \node[fill=black, rectangle, inner sep=0pt,minimum size=7pt] at (180:1cm) (1) {};
          \node[fill=black, rectangle, inner sep=0pt,minimum size=7pt] at (270:1cm) (2) {};
          \node[fill=black, rectangle, inner sep=0pt,minimum size=7pt] at (0:1cm) (3) {};
          \node[fill=black, circle, inner sep=0pt,minimum size=7pt] at (90:1cm) (4) {};
          \draw[->] (1) to (2);
          \draw[->] (2) to (3);
          \draw[->] (3) to (1);
          \draw[->] (4) to (1);
          \draw[->] (4) to (2);
          \draw[->] (3) to (4);
          \node (label) at (270:1.5cm) {$F^3_2$};
        \end{scope}
        \begin{scope}[xshift=6cm, yshift=0.0cm]
          \node[fill=black, rectangle, inner sep=0pt,minimum size=7pt] at (180:1cm) (1) {};
          \node[fill=black, rectangle, inner sep=0pt,minimum size=7pt] at (270:1cm) (2) {};
          \node[fill=black, rectangle, inner sep=0pt,minimum size=7pt] at (0:1cm) (3) {};
          \node[fill=black, circle, inner sep=0pt,minimum size=7pt] at (90:1cm) (4) {};
          \draw[->] (1) to (2);
          \draw[->] (2) to (3);
          \draw[->] (3) to (1);
          \draw[->] (4) to (1);
          \draw[->] (2) to (4);
          \draw[->] (4) to (3);
          \node (label) at (270:1.5cm) {$F^3_3$};
        \end{scope}
        \begin{scope}[xshift=9cm, yshift=0.0cm]
          \node[fill=black, rectangle, inner sep=0pt,minimum size=7pt] at (180:1cm) (1) {};
          \node[fill=black, rectangle, inner sep=0pt,minimum size=7pt] at (270:1cm) (2) {};
          \node[fill=black, rectangle, inner sep=0pt,minimum size=7pt] at (0:1cm) (3) {};
          \node[fill=black, circle, inner sep=0pt,minimum size=7pt] at (90:1cm) (4) {};
          \draw[->] (1) to (2);
          \draw[->] (2) to (3);
          \draw[->] (3) to (1);
          \draw[->] (4) to (1);
          \draw[->] (2) to (4);
          \draw[->] (3) to (4);
          \node (label) at (270:1.5cm) {$F^3_4$};
        \end{scope}
        \begin{scope}[xshift=0cm, yshift=-3cm]
          \node[fill=black, rectangle, inner sep=0pt,minimum size=7pt] at (180:1cm) (1) {};
          \node[fill=black, rectangle, inner sep=0pt,minimum size=7pt] at (270:1cm) (2) {};
          \node[fill=black, rectangle, inner sep=0pt,minimum size=7pt] at (0:1cm) (3) {};
          \node[fill=black, circle, inner sep=0pt,minimum size=7pt] at (90:1cm) (4) {};
          \draw[->] (1) to (2);
          \draw[->] (2) to (3);
          \draw[->] (3) to (1);
          \draw[->] (1) to (4);
          \draw[->] (4) to (2);
          \draw[->] (4) to (3);
          \node (label) at (270:1.5cm) {$F^3_5$};
        \end{scope}
        \begin{scope}[xshift=3cm, yshift=-3cm]
          \node[fill=black, rectangle, inner sep=0pt,minimum size=7pt] at (180:1cm) (1) {};
          \node[fill=black, rectangle, inner sep=0pt,minimum size=7pt] at (270:1cm) (2) {};
          \node[fill=black, rectangle, inner sep=0pt,minimum size=7pt] at (0:1cm) (3) {};
          \node[fill=black, circle, inner sep=0pt,minimum size=7pt] at (90:1cm) (4) {};
          \draw[->] (1) to (2);
          \draw[->] (2) to (3);
          \draw[->] (3) to (1);
          \draw[->] (1) to (4);
          \draw[->] (4) to (2);
          \draw[->] (3) to (4);
          \node (label) at (270:1.5cm) {$F^3_6$};
        \end{scope}
        \begin{scope}[xshift=6cm, yshift=-3cm]
          \node[fill=black, rectangle, inner sep=0pt,minimum size=7pt] at (180:1cm) (1) {};
          \node[fill=black, rectangle, inner sep=0pt,minimum size=7pt] at (270:1cm) (2) {};
          \node[fill=black, rectangle, inner sep=0pt,minimum size=7pt] at (0:1cm) (3) {};
          \node[fill=black, circle, inner sep=0pt,minimum size=7pt] at (90:1cm) (4) {};
          \draw[->] (1) to (2);
          \draw[->] (2) to (3);
          \draw[->] (3) to (1);
          \draw[->] (1) to (4);
          \draw[->] (2) to (4);
          \draw[->] (4) to (3);
          \node (label) at (270:1.5cm) {$F^3_7$};
        \end{scope}
        \begin{scope}[xshift=9cm, yshift=-3cm]
          \node[fill=black, rectangle, inner sep=0pt,minimum size=7pt] at (180:1cm) (1) {};
          \node[fill=black, rectangle, inner sep=0pt,minimum size=7pt] at (270:1cm) (2) {};
          \node[fill=black, rectangle, inner sep=0pt,minimum size=7pt] at (0:1cm) (3) {};
          \node[fill=black, circle, inner sep=0pt,minimum size=7pt] at (90:1cm) (4) {};
          \draw[->] (1) to (2);
          \draw[->] (2) to (3);
          \draw[->] (3) to (1);
          \draw[->] (1) to (4);
          \draw[->] (2) to (4);
          \draw[->] (3) to (4);
          \node (label) at (270:1.5cm) {$F^3_8$};
        \end{scope}
    \end{tikzpicture}
    \end{center}

Next, we require a multiplication operation on pairs of flags.

\begin{defn}
Let $F$ be a tournament with vertex set $\{1,\dots,r\}$. Say that two flags $(F_1,r)$ and $(F_2,r)$ are \emph{$R$-compatible} if $v(F_1)=v(F_2)$ and, for each $i\in\{1,2\}$, the subtournament of $F_i$ induced by the roots is equal to $R$. Say that $(F_1,r)$ and $(F_2,r)$ are \emph{compatible} if they are $R$-compatible for some tournament $R$ of order $r$.
\end{defn}

\begin{defn}
Let $F$ be a tournament with vertex set $\{1,\dots,r\}$. Given $R$-compatible flags $(F_1,r)$ and $(F_2,r)$ and a tournamenton $W$, define $t_r(F_1\cdot F_2,W):[0,1]^r\to[0,1]$ by
\[t_r(F_1\cdot F_2,W)(x_1,\dots,x_r):=\begin{cases}\frac{t_r(F_1,W)(x_1,\dots,x_r)\cdot t_r(F_2,W)(x_1,\dots,x_r)}{t_r(R,W)(x_1,\dots,x_r)}& \text{if }t_r(R,W)(x_1,\dots,x_r)\neq 0,\\ 0&\text{otherwise}.\end{cases}\]
\end{defn}

A key idea in the flag algebra method is that, for two $R$-compatible flags $(F_1,r)$ and $(F_2,r)$, the integral of $t_r(F_1\cdot F_2,W)(x_1,\dots,x_r)$ over $[0,1]^r$ can be expressed as a linear combination of $d(J,W)$ over all tournaments $J$ with $v(F_1)+v(F_2)-v(F)$ vertices. We have actually already seen an instance of this, in disguise, in the proof of Proposition~\ref{prop:regularTransitiveW}. If we let $F_1=F_2$ be the tournament on vertex set $\{1,2\}$ where $1$ is the source, then $t_1(F_1\cdot F_2,W)(x_1)=\left(d_W^+(x_1)\right)^2$. Therefore, \eqref{eq:outdeg} expresses $\int_0^1t_1(F_1\cdot F_2,W)(x_1)dx_1$ as a linear combination of $d(J,W)$ over all tournaments $J$ with $3$ vertices. The same sort of ``handshaking'' argument applies to any other product of compatible flags to yield the following lemma.

\begin{lem}
\label{lem:flagMult}
Let $F$ be a tournament with vertex set $\{1,\dots,r\}$, let $(F_1,r)$ and $(F_2,r)$ be $R$-compatible flags and let $m\geq v(F_1)+v(F_2)-v(F)$. Then there exist constants $b_r(F_1,F_2;J)$ for each tournament $J$ on $m$ vertices such that 
\[\int_0^1\cdots\int_0^1t_r(F_1\cdot F_2,W)(x_1,\dots,x_r)dx_1\cdots dx_r = \sum_{J:v(J)=m}b_r(F_1,F_2;J)\cdot d(J,W)\]
for every tournamenton $W$, where the sum on the right side is over all $m$-vertex tournaments up to isomorphism.
\end{lem}

We provide an example to illustrate how to compute the coefficients $b_r(F_1,F_2;J)$ in Lemma~\ref{lem:flagLem}. We also remark that all of the coefficients that are needed for our proofs are listed in Appendix~\ref{app:coeffs} for convenience. Consider the flags $F_1^1$ and $F_3^1$ from $\mathcal{F}_3(TT_2)$ as defined earlier in this section. Then $\int_0^1\int_0^1t(F_1^1\cdot F_3^1,W)(x_1,x_2)dx_1dx_2$ can be written as
\[\int_0^1\int_0^1W(x_2,x_1)\left(\int_0^1W(x_3,x_1)W(x_3,x_2)dx_3\right)\left(\int_0^1W(x_3,x_2)W(x_1,x_3)dx_3\right)dx_1dx_2\]
which, by a change of variables, is 
\[=\int_0^1\int_0^1W(x_2,x_1)\left(\int_0^1W(x_3,x_1)W(x_3,x_2)dx_3\right)\left(\int_0^1W(x_4,x_2)W(x_1,x_4)dx_4\right)dx_1dx_2\]
\[=\int_0^1\int_0^1\int_0^1\int_0^1W(x_2,x_1)W(x_3,x_1)W(x_3,x_2)W(x_4,x_2)W(x_1,x_4)dx_1dx_2dx_3dx_4\]
\[=\int_{[0,1]^4}W(x_2,x_1)W(x_3,x_1)W(x_3,x_2)W(x_4,x_2)W(x_1,x_4)\left[W(x_4,x_3)+W(x_3,x_4)\right]dx_1dx_2dx_3dx_4.\]
Expanding the sum in the above expression yields
\begin{equation}\label{eq:example}\int_0^1\int_0^1t(F_1^1\cdot F_3^1,W)(x_1,x_2)dx_1dx_2=t(H_5,W) +t(H_6,W).\end{equation}
Now, by definition of $d(\cdot,W)$, the right side of \eqref{eq:example} can be rewritten as 
\[\frac{3}{24}d(H_5,W) + \frac{1}{24}d(H_6,W).\]
Thus, for these two flags, we have $b_2(F_1^1,F_3^1;H_5)=\frac{3}{24}$ and $b_2(F_1^1,F_3^1;H_6)=\frac{1}{24}$. Also, $b_2(F_1^1,F_3^1;J)=0$ for any $4$-vertex tournament not isomorphic to $H_5$ or $H_6$. For any flags $(F_1,r)$ and $(F_2,r)$ and tournamant $J$ with at least $v(F_1)+v(F_2)-r$ vertices, the coefficient $b_r(F_1,F_2;J)$ in Lemma~\ref{lem:flagMult} can be computed explicitly via a similar procedure. As mentioned earlier, all of the coefficients that we require in our proofs are included in Appendix~\ref{app:coeffs}. 

Finally, we state the key lemma that will be used in the proofs of Theorems~\ref{th:H10}--\ref{th:H14}. A symmetric $t\times t$ matrix $M$ is said to be \emph{positive semidefinite}, written $M\succcurlyeq 0$, if all of its eigenvalues are non-negative. Equivalently, $M\succcurlyeq0$ if and only if $\sum_{i=1}^t\sum_{j=1}^tM_{i,j}x_ix_j\geq 0$ for any $x_1,\dots,x_t\in\mathbb{R}$. 

\begin{lem}
\label{lem:flagLem}
Let $m$ be an integer and $H\in\mathbb{R}[\mathcal{T}_{\leq m}]$. Let $\ell,t_1,\dots,t_\ell,r_1,\dots,r_\ell,k_1,\dots,k_\ell$ be positive integers such that $k_q\geq r_q+1$ and $2k_q-r_q\leq m$ for all $1\leq q\leq \ell$. For each $1\leq q\leq\ell$, let $(F^q_1,r_q),\dots,(F^q_{t_q},r_q)$ be pairwise compatible flags on vertex set $\{1,\dots,k_q\}$ and $A_q$ be a positive semidefinite $t_q\times t_q$ matrix. Then every tournamenton $W$ satisfies
\[t(H,W)\geq \min_{J: v(J)=m}\left\{t_{\inj}(H,J) -\sum_{q=1}^\ell\sum_{i=1}^{t_q}\sum_{j=1}^{t_q}b_{r_q}(F_i\cdot F_j;J)\cdot (A_q)_{i,j}\right\}.\]
Moreover, if $W$ is a tournamenton such that equality holds, then, for every $1\leq q\leq \ell$ and almost every $(x_1,\dots,x_{r_q})\in[0,1]^{r_q}$, the vector
\[\left(t_{r_q}(F_1^q,W)(x_1,\dots,x_{r_q}),\dots,t_{r_q}(F_{t_q}^q,W)(x_1,\dots,x_{r_q})\right)\]
is in the kernel of $A_q$. 
\end{lem}

\begin{proof}
Let $W$ be any tournamenton. By Lemma~\ref{lem:t to tinj}, 
\begin{equation}\label{eq:ttotinj}
t(H,W) = \sum_{J: v(J)=m}t_{\inj}(H,J)d(J,W).
\end{equation}
We claim that the following holds for any $1\leq q\leq \ell$ and $(x_1,\dots,x_{r_q})\in[0,1]^{r_q}$:
\begin{equation}\label{eq:psd}\sum_{i=1}^{t_q} \sum_{j=1}^{t_q}(A_q)_{i,j}\cdot t(F_i^q\cdot F_j^q,W)(x_1,\dots,x_{r_q}) \geq 0.\end{equation}
Indeed, if $(x_1,\dots,x_{r_q})\in[0,1]^{r_q}$ are such that $t(R^q,W)(x_1,\dots,x_{r_q})=0$, where $R^q$ is the subgraph of any of the flags $F_1^q,\dots,F_{t_q}^q$ induced by $\{1,\dots,r_q\}$, then every term of the sum is zero. On the other hand, if $t(R^q,W)(x_1,\dots,x_{r_q})\neq 0$, then applying the definition of $t(F_i^q\cdot F_j^q,W)(x_1,\dots,x_{r_q})$ gives us that the sum is non-negative due to the fact that $A_q$ is positive semidefinite. Now, if we integrate out all of the variables in \eqref{eq:psd}, we get 
\[\sum_{J: v(J)=m}\sum_{q=1}^\ell\sum_{i=1}^{t_q}\sum_{j=1}^{t_q}(A_q)_{i,j}\cdot b_{r_q}(F_i\cdot F_j;J)\cdot d(J,W) \geq 0.\]
Putting this together with \eqref{eq:ttotinj} gives us 
\[t(H,W)\geq \sum_{J: v(J)=m}\left(t_{\inj}(H,J) - \sum_{q=1}^\ell\sum_{i=1}^{t_q}\sum_{j=1}^{t_q}(A_q)_{i,j}\cdot b_{r_q}(F_i\cdot F_j;J)\right)\cdot d(J,W).\]
Finally, by \eqref{lem:sumToOne}, we have that $\sum_{J: v(J)=m}d(J,W)=1$ and so the quantity on the right side of the above inequality is bounded below by the minimum of the coefficient of $d(J,W)$ over all $J$. This is precisely the bound in the lemma. 

Now, let us discuss the case of equality. If, for some $1\leq q\leq \ell$, there is a positive measure subset of $[0,1]^{r_q}$ such that the vector
\[\left(t_{r_q}(F_1^q,W)(x_1,\dots,x_{r_q}),\dots,t_{r_q}(F_{t_q}^q,W)(x_1,\dots,x_{r_q})\right)\]
is not in the kernel of $A_q$, then the inequality \eqref{eq:psd} is strict for a positive measure of points. This, in turn, translates to a gap in the final bound. So, if equality holds, then such vectors must be in the kernel of $A_q$ for each $1\leq q\leq\ell$ and almost all $(x_1,\dots,x_{r_q})$. 
\end{proof}

Without further delay, we use the flag algebra method to prove Theorems~\ref{th:H10}--\ref{th:H14}.

\begin{proof}[Proof of Theorem~\ref{th:H10}]
We apply Lemma~\ref{lem:flagLem} with the following matrices:
\[A_2= \frac{1}{245}\begin{bmatrix}
4724&  -1883& 1081&  -4598& 3827&  -293&  1390&  -4248\\
-1883& 6512&  -4787& 3150&   39&    378&  -2965& -444 \\
1081&  -4787& 3856&  -2274& -371&  -313&  2010&   798 \\
-4598& 3150&  -2274& 5490&  -3658&  156&  -1734& 3468 \\
3827&   39&   -371&  -3658& 4600&  -420&   945&  -4962\\
-293&   378&  -313&   156&  -420&   606&  -660&   546 \\
1390&  -2965& 2010&  -1734&  945&  -660&  2376&  -1362\\
-4248& -444&   798&  3468&  -4962&  546&  -1362& 6204 
\end{bmatrix}\]
\[A_3= \frac{1}{245}\begin{bmatrix}
  612&    162&    162&   -162&    162&   -162&   -162&   -612 \\
  162&   2082&    342&    252&    342&    -24&   -1050&  -2106\\
  162&    342&   2082&    -24&    342&   -1050&   252&   -2106\\
 -162&    252&    -24&   1554&   -1050&  -186&   -186&   -198 \\
  162&    342&    342&   -1050&  2082&    252&    -24&   -2106\\
 -162&    -24&   -1050&  -186&    252&   1554&   -186&   -198 \\
 -162&   -1050&   252&   -186&    -24&   -186&   1554&   -198 \\
 -612&   -2106&  -2106&  -198&   -2106&  -198&   -198&   7524   \end{bmatrix}\]
Both of these matrices are positive semidefinite. We need that, for every tournament $J$ on $5$ vertices, the expression
\[-8\cdot t_{\inj}(C_3,J)-\frac{1}{4}\cdot1024\cdot t_{\inj}(H_{10},J) -\sum_{i=1}^8\sum_{j=1}^8b_2(F^2_i\cdot F^2_j; J)\cdot (A_2)_{i,j} -\sum_{i=1}^8\sum_{j=1}^8b_2(F^3_i\cdot F^3_j; J)\cdot (A_3)_{i,j}\]
is equal to $-5/4$. The values of the coefficients can be found in the appendix at the end of the paper. We will only compute this expression for two examples of $J$ for the sake of demonstration. First, in the case that $J=H_8$, it becomes
\[-8\cdot 0-\frac{1}{4}\cdot1024\cdot 0 - \frac{2}{120}(A_2)_{1,1}-\frac{1}{120}(A_2)_{1,2}-\frac{1}{120}(A_2)_{1,4}-\frac{1}{120}(A_2)_{1,8}-\frac{1}{120}(A_2)_{2,1}-\frac{2}{120}(A_2)_{2,2}\]
\[-\frac{1}{120}(A_2)_{2,4}-\frac{1}{120}(A_2)_{2,8}-\frac{1}{120}(A_2)_{4,1}-\frac{1}{120}(A_2)_{4,2}-\frac{2}{120}(A_2)_{4,4}-\frac{1}{120}(A_2)_{4,8}-\frac{1}{120}(A_2)_{8,1}-\frac{1}{120}(A_2)_{8,2}\]
\[-\frac{1}{120}(A_2)_{8,4}-\frac{2}{120}(A_2)_{8,8}\]
which, upon plugging the appropriate values from $A_2$, becomes
\[-8\cdot 0-\frac{1}{4}\cdot1024\cdot 0 - \frac{2}{120}\cdot\frac{4724}{245}+\frac{1}{120}\cdot\frac{1883}{245}+\frac{1}{120}\cdot\frac{4598}{245}+\frac{1}{120}\cdot\frac{4248}{245}+\frac{1}{120}\cdot\frac{1883}{245}-\frac{2}{120}\cdot\frac{6512}{245}\]
\[-\frac{1}{120}\cdot\frac{3150}{245}+\frac{1}{120}\cdot\frac{444}{245}+\frac{1}{120}\cdot\frac{4598}{245}-\frac{1}{120}\cdot\frac{3150}{245}-\frac{2}{120}\cdot\frac{5490}{245}-\frac{1}{120}\cdot\frac{3468}{245}+\frac{1}{120}\cdot\frac{4248}{245}+\frac{1}{120}\cdot\frac{444}{245}\]
\[-\frac{1}{120}\cdot\frac{3468}{245}-\frac{2}{120}\cdot\frac{6204}{245}=-\frac{5}{4}\]
As another example, consider $J=H_{12}$. The expression becomes 
\[-8\cdot \frac{12}{60}-\frac{1}{4}\cdot1024\cdot 0 - \frac{3}{120}(A_2)_{3,5}-\frac{3}{120}(A_2)_{5,3}-\frac{3}{120}(A_2)_{6,7}-\frac{3}{120}(A_2)_{7,6}-\frac{3}{120}(A_3)_{1,8}\]
\[-\frac{3}{120}(A_3)_{2,6}-\frac{3}{120}(A_3)_{3,4}-\frac{3}{120}(A_3)_{4,3}-\frac{3}{120}(A_3)_{5,7}-\frac{3}{120}(A_3)_{6,2}-\frac{3}{120}(A_3)_{7,5}-\frac{3}{120}(A_3)_{8,1}\]
which evaluates to
\[-8\cdot \frac{12}{60}-\frac{1}{4}\cdot1024\cdot 0 + \frac{3}{120}\cdot\frac{371}{245}+\frac{3}{120}\cdot\frac{371}{245}+\frac{3}{120}\cdot\frac{660}{245}+\frac{3}{120}\cdot\frac{660}{245}+\frac{3}{120}\cdot\frac{612}{245}\]
\[+\frac{3}{120}\cdot\frac{24}{245}+\frac{3}{120}\cdot\frac{24}{245}+\frac{3}{120}\cdot\frac{24}{245}+\frac{3}{120}\cdot\frac{24}{245}+\frac{3}{120}\cdot\frac{24}{245}+\frac{3}{120}\cdot\frac{24}{245}+\frac{3}{120}\cdot\frac{612}{245}=-\frac{5}{4}.\]

Now, suppose that $W$ is a tournamenton such that
\[8\cdot t(C_3,W) + \frac{1}{4}\cdot1024\cdot t(H_{10},W) =\frac{5}{4}\]
and let us show that $W$ is equal to $1/2$ almost everywhere. The kernel of $A_2$ is spanned by $(1,1,1,1,1,1,1,1)^T$. By the ``moreover'' part of Lemma~\ref{lem:flagLem}, we must have that 
\[\label{eq:allSameT}t_3(F_1^2,W)(x_1,x_2,x_3)=t_3(F_2^2,W)(x_1,x_2,x_3)=\cdots =t_3(F_8^2,W)(x_1,x_2,x_3).\]
By Lemma~\ref{lem:sumAllOrientations}, we have
\[\sum_{i=1}^8t(F_i^2,W)(x_1,x_2,x_3) = t(TT_3,W)(x_1,x_2,x_3)\] 
for almost all $(x_1,x_2,x_3)\in [0,1]^3$. Putting these two facts together, we must have $t_3(F_1^2,W)(x_1,x_2,x_3)=\frac{1}{8}t(TT_3,W)(x_1,x_2,x_3)$ and $t_3(F_7^2,W)(x_1,x_2,x_3)=\frac{1}{8}t(TT_3,W)(x_1,x_2,x_3)$ 
for almost all $(x_1,x_2,x_3)\in[0,1]^3$. Integrating out these two equalities over all $(x_1,x_2,x_3)\in[0,1]^3$ yields $t(TT_4,W)=\frac{1}{8}t(TT_3,W)$ and $t(C_4,W)=\frac{1}{8}t(TT_3,W)$, respectively. So, we get that $W=1/2$ almost everywhere by Lemma~\ref{lem:uniqueness}.
\end{proof}

\begin{proof}[Proof of Theorem~\ref{th:H11}]
We apply Lemma~\ref{lem:flagLem} with the following matrices:
\[A_2= \frac{1}{7}\begin{bmatrix}
257 & -130 & -214 &  127 &  -57 &  201 &  40 &  -224 \\
-130 &  199 &  123 & -179 &  72 &  -130 &  -86 &  131 \\
-214 &  123 &  210 & -134 &  45 &  -190 &  -48 &  208 \\
127 & -179 & -134 &  204 &  -79 &  120 &  77 &  -136 \\
-57 &  72 &   45 &   -79 &  58 &   -44 &  -30 &  35  \\
201 & -130 & -190 &  120 &  -44 &  219 &  51 &  -227 \\
40 &   -86 &  -48 &  77 &   -30 &  51 &   59 &   -63 \\
-224 &  131 &  208 & -136 &  35 &  -227 &  -63 &  276
\end{bmatrix}\]
\[A_3= \frac{1}{7}\begin{bmatrix}
184& -31& -31& 22&  -31& 22&  22& -157 \\
-31& 62&  -10& -26& -10&  4&  -11& 22  \\
-31& -10& 62&   4&  -10& -11& -26& 22  \\
22&  -26&  4&  61&  -11& -11& -11& -28 \\
-31& -10& -10& -11& 62&  -26&  4&  22  \\
22&   4&  -11& -11& -26& 61&  -11& -28 \\
22&  -11& -26& -11&  4&  -11& 61&  -28 \\
-157& 22&  22&  -28& 22&  -28& -28& 175   \end{bmatrix}\]

These matrices are positive semidefinite. We need that, for every tournament $J$ on $5$ vertices, the expression
\[-8\cdot t_{\inj}(C_3,J)-\frac{1}{4}\cdot1024\cdot t_{\inj}(H_{11},J) -\sum_{i=1}^8\sum_{j=1}^8b_2(F^2_i\cdot F^2_j; J)\cdot (A_2)_{i,j} -\sum_{i=1}^8\sum_{j=1}^8b_2(F^3_i\cdot F^3_j; J)\cdot (A_3)_{i,j}\]
is equal to $-5/4$. The values of the coefficients can be found in the appendix at the end of the paper. We will only compute this expression in one example for the sake of demonstration. In the case $J=H_{19}$, it becomes
\[-8\cdot \frac{15}{60}-\frac{1}{4}\cdot1024\cdot 0 - \frac{5}{120}(A_2)_{5,7} - \frac{5}{120}(A_2)_{7,5}-\frac{5}{120}(A_3)_{2,7}-\frac{5}{120}(A_3)_{3,6}-\frac{5}{120}(A_3)_{4,5}-\frac{5}{120}(A_3)_{5,4}\]
\[-\frac{5}{120}(A_3)_{6,3}-\frac{5}{120}(A_3)_{7,2}\]
which is
\[-8\cdot \frac{15}{60}-\frac{1}{4}\cdot1024\cdot 0 + \frac{5}{120}\cdot\frac{30}{7} + \frac{5}{120}\cdot\frac{30}{7}+\frac{5}{120}\cdot\frac{11}{7}+\frac{5}{120}\cdot\frac{11}{7}+\frac{5}{120}\cdot\frac{11}{7}+\frac{5}{120}\cdot\frac{11}{7}\]
\[+\frac{5}{120}\cdot\frac{11}{7}+\frac{5}{120}\cdot\frac{11}{7}=-\frac{5}{4}.\]

Now, suppose that $W$ is a regular tournamenton such that
\[8\cdot t(C_3,W) + \frac{1}{4}\cdot1024\cdot t(H_{11},W) =\frac{5}{4}\]
and let us show that $W$ is equal to $1/2$ almost everywhere. The kernel of $A_2$ is spanned by $(1,1,1,1,1,1,1,1)^T$. So, analogously to the proof of Theorem~\ref{th:H10}, we get that $W$ is equal to $1/2$ almost everywhere.
\end{proof}

\begin{proof}[Proof of Theorem~\ref{th:H13}]
We apply Lemma~\ref{lem:flagLem} with the following matrix:
\[A_2=\frac{1}{945}\begin{bmatrix}
3680&   1296&   1080&   -2376&  -6376&  1400&   2808&   -1512 \\
1296&   4528&   -1080&  -1512&  -4424&  -1400&  5832&   -3240 \\
1080&   -1080&  5616&   -2160&  -3240&  3240&   -5400&  1944  \\
-2376&  -1512&  -2160&  4320&   5400&   -1512&  -3240&  1080  \\
-6376&  -4424&  -3240&  5400&   16200&  -5400&  -5400&  3240  \\
1400&   -1400&  3240&   -1512&  -5400&  5400&   -3240&  1512  \\
2808&   5832&   -5400&  -3240&  -5400&  -3240&  16200&  -7560 \\
-1512&  -3240&  1944&   1080&   3240&   1512&   -7560&  4536  \\

\end{bmatrix}.\]
This matrix is positive semidefinite. We need that, for every tournament $J$ on $5$ vertices, the expression
\[8\cdot t_{\inj}(TT_3,J)+\frac{1}{7}\cdot1024\cdot t_{\inj}(H_{13},J) -\sum_{i=1}^8\sum_{j=1}^8b_2(F^2_i\cdot F^2_j; J)\cdot (A_2)_{i,j}\]
is equal to $8/7$. The values of the coefficients can be found in the appendix at the end of the paper. We will only compute this expression in one example for the sake of demonstration. In the case $J=H_{13}$, we get
\[8\cdot \frac{6}{60}+\frac{1}{7}\cdot1024\cdot \frac{1}{120} - \frac{1}{120}(A_2)_{1,7} - \frac{1}{120}(A_2)_{2,7}-\frac{1}{120}(A_2)_{4,5}-\frac{1}{120}(A_2)_{5,4}\]
\[-\frac{2}{120}(A_2)_{5,5}-\frac{1}{120}(A_2)_{5,8}-\frac{1}{120}(A_2)_{7,1}-\frac{1}{120}(A_2)_{7,2}-\frac{2}{120}(A_2)_{7,7}-\frac{1}{120}(A_2)_{8,5}\]
and plugging in yields
\[8\cdot \frac{6}{60}+\frac{1}{7}\cdot1024\cdot \frac{1}{120} - \frac{1}{120}\cdot\frac{2808}{945} - \frac{1}{120}\cdot\frac{5832}{945}-\frac{1}{120}\cdot\frac{5400}{945}-\frac{1}{120}\cdot\frac{5400}{945}\]
\[-\frac{2}{120}\cdot\frac{16200}{945}-\frac{1}{120}\cdot\frac{3240}{945}-\frac{1}{120}\cdot\frac{2808}{945}-\frac{1}{120}\cdot\frac{5832}{945}-\frac{2}{120}\cdot\frac{16200}{945}-\frac{1}{120}\cdot\frac{3240}{945}=\frac{8}{7}.\]

Now, suppose that $W$ is a tournamenton such that
\[8\cdot t(TT_3,W) + \frac{1}{7}\cdot1024\cdot t(H_{13},W) =\frac{8}{7}\]
and let us show that $W$ is equal to $1/2$ almost everywhere. The kernel of $A_2$ is spanned by $(1,1,1,1,1,1,1,1)^T$ and $(1,1,-1,-1,1,1,-1,-1)^T$. By Lemma~\ref{lem:sumAllOrientations}, for almost every $(x_1,x_2,x_3)\in[0,1]^3$, we have 
\[\sum_{i=1}^8t(F_i^2,W)(x_1,x_2,x_3)=t(TT_3,W)(x_1,x_2,x_3).\] 
Combining these two facts with the ``moreover'' part of Lemma~\ref{lem:flagLem}, we must have that, for almost every $(x_1,x_2,x_3)\in[0,1]^3$ and any $(i,j)\in\{1,2,4,5\}\times\{3,4,7,8\}$, 
\[t(F_i^2,W)(x_1,x_2,x_3)+t(F_j^2,W)(x_1,x_2,x_3)=\frac{1}{4}t(TT_3,W)(x_1,x_2,x_3).\]
Applying this in the case $i=5$ and $j=7$ and integrating over all $(x_1,x_2,x_3)$ yields 
\[\label{eq:C4TT3}t(C_4,W)=\frac{1}{8}t(TT_3,W).\]
Also, setting $i=1$ and $j=4$ and integrating gives us 
\[\label{eq:TT4TT3}t(TT_4,W)=\frac{1}{8}t(TT_3,W).\]
So, we conclude that $W=1/2$ almost everywhere by Lemma~\ref{lem:uniqueness}.
\end{proof}

\begin{proof}[Proof of Theorem~\ref{th:H14}]
We apply Lemma~\ref{lem:flagLem} with the following matrices:
\[A_1 = \frac{4}{5}\begin{bmatrix}
1 & -1 & -1 & 1\\
-1 & 1 & 1 & -1\\
-1 & 1 & 1 & -1\\
1 & -1 & -1 & 1
\end{bmatrix}\]
and
\[A_3 = \frac{8}{15}\begin{bmatrix}
3 & -1 & -1 & -1 & -1 & -1 & -1 & 3\\
-1 & 27 & -5 & -5 & -5 & -5 & -5 & -1\\
-1 & -5 & 27 & -5 & -5 & -5 & -5 & -1\\
-1 & -5 & -5 & 27 & -5 & -5 & -5 & -1\\
-1 & -5 & -5 & -5 & 27 & -5 & -5 & -1\\
-1 & -5 & -5 & -5 & -5 & 27 & -5 & -1\\
-1 & -5 & -5 & -5 & -5 & -5 & 27 & -1\\
3 & -1 & -1 & -1 & -1 & -1 & -1 & 3
\end{bmatrix}.\]
Both are positive semidefinite. We need that, for every tournament $J$ on $5$ vertices, the expression
\[8\cdot t_{\inj}(TT_3,J)+\frac{1}{5}\cdot1024\cdot t_{\inj}(H_{14},J) -\sum_{i=1}^4\sum_{j=1}^4b_2(F^1_i\cdot F^1_j; J)\cdot (A_1)_{i,j}-\sum_{i=1}^8\sum_{j=1}^8b_2(F^3_i\cdot F^3_j; J)\cdot (A_3)_{i,j}\]
is equal to $6/5$. Again, the values of the coefficients can be found in the appendix at the end of the paper. For $J=H_8$, the expression becomes
\[8\cdot \frac{10}{60} + \frac{1}{15}\cdot 1024\cdot 0 - \frac{10}{120}\cdot (A_1)_{1,1} - \frac{10}{120}\cdot (A_1)_{1,2}-\frac{10}{120}\cdot (A_1)_{1,4}-\frac{10}{120}\cdot (A_1)_{2,2}\]
\[-\frac{10}{120}\cdot (A_1)_{2,4}-\frac{10}{120}\cdot (A_1)_{4,4}\]
\[=\frac{4}{3}+0-\frac{1}{15}+\frac{1}{15}-\frac{1}{15}-\frac{1}{15}+\frac{1}{15}-\frac{1}{15}=\frac{6}{5}.\]
For $J=H_9$, it is 
\[8\cdot \frac{9}{60} + \frac{1}{15}\cdot 1024\cdot 0 - \frac{6}{120}A_1(1,1) -\frac{18}{120}A_1(1,2) - \frac{6}{120}A_1(1,4) - \frac{6}{120}A_1(2,2)\]
\[- \frac{6}{120}A_1(2,4)-\frac{12}{120}A_1(3,4)-\frac{6}{120}A_1(4,4)-\frac{6}{120}A_3(8,8)\]
\[=\frac{6}{5} + 0 -\frac{1}{25}+\frac{3}{25}-\frac{1}{25}-\frac{1}{25}+\frac{1}{25}+\frac{2}{25}-\frac{1}{25}-\frac{2}{25}=\frac{6}{5}.\]
The rest of the calculations are similar.

Now, suppose that $W$ is a tournamenton such that 
\[8\cdot t(TT_3,W) + \frac{1}{5}\cdot1024\cdot t(H_{14},W) = \frac{6}{5}.\]
The kernel of $A_1$ is spanned by $(1,1,0,0)^T,(1,0,1,0)^T$ and $(0,1,0,1)^T$. By the ``moreover'' part of Lemma~\ref{lem:flagLem}, this implies that 
\[t_2(F_2^1,W)(x_1,x_2)+t_2(F_3^1,W)(x_1,x_2)=\frac{1}{2}t(H_1,W)(x_1,x_2) = \frac{1}{2}W(x_1,x_2)\]
for almost all $(x_1,x_2)\in [0,1]^2$. Integrating this out over all $x_1$ and $x_2$ yields 
\[t(TT_3,W)+t(C_3,W) = \frac{1}{4}\]
which, when combined with \eqref{eq:T3C3}, tells us that 
\begin{equation}\label{eq:itsRegular}t(C_3,W)=1/8.\end{equation} 
The kernel of $A_3$ is spanned by $(1,1,1,1,1,1,1,1)^T$ and $(1,0,0,0,0,0,0,-1)^T$. Applying Lemma~\ref{lem:sumAllOrientations} and the ``moreover'' part of Lemma~\ref{lem:flagLem} again, this implies that $t(F_2^3,W)(x_1,x_2,x_3)=\frac{1}{8}t(C_3,W)(x_1,x_2,x_3)$  and  $t(F_1^3,W)(x_1,x_2,x_3)+t(F_8^3,W)(x_1,x_2,x_3)=\frac{1}{4}t(C_3,W)(x_1,x_2,x_3)$ for almost all $(x_1,x_2,x_3)\in[0,1]^3$. Integrating over all $x_1,x_2$ and $x_3$ and applying \eqref{eq:itsRegular} yields
\[t(C_4,W)=\frac{1}{8}t(C_3,W) = \frac{1}{64}\]
and
\[t(H_5,W)+t(H_7,W) = \frac{1}{4}t(C_3,W)=\frac{1}{32}.\]
Finally, using Lemma~\ref{lem:sumToOne} and \eqref{eq:dt} and the fact that $\aut(TT_4)=\aut(C_4)=1$ and $\aut(H_5)=\aut(H_7)=3$, we get
\[d(TT_4,W)+d(C_4,W)+d(H_5,W)+d(H_7,W)=1\]
\[\Longrightarrow 24t(TT_4,W) + 24t(C_4,W) + 8t(H_5,W)+8t(H_7,W)=1\]
which, since $t(C_4,W)=1/64$ and $t(H_5,W)+t(H_7,W)=1/32$, yields $t(TT_4,W)=1/64$. We are now done by the fact that $TT_4$ forces quasirandomness (see Theorem~\ref{th:qriff}). 
\end{proof}

\section{Conclusion}
\label{sec:concl}

The Forcing Conjecture~\cite{ConlonFoxSudakov10,SkokanThoma04} says that a graph $H$ forces quasirandomness if and only if it is bipartite and contains a cycle. Unfortunately, the results of this paper do not seem to point to such a natural property which distinguishes the tournaments which force quasirandomness in regular tournaments in general. Nonetheless, we wonder whether a deeper investigation into such tournaments could uncover some sort of ``pattern.''

\begin{prob}
\label{prob:classify}
Classify tournaments $H$ which force quasirandomness in regular tournaments. 
\end{prob}

One way in which Problem~\ref{prob:classify} could end up being solved is if, as in the case of tournaments which force quasirandomness in general (not necessarily regular) tournaments, there end up being only finitely many non-transitive tournaments which force quasirandomness in regular tournaments. So, a first step toward Problem~\ref{prob:classify} is to determine whether or not an infinite family of such tournaments exists. 

\begin{ques}
\label{ques:infinitelyMany}
Are there infinitely many non-transitive tournaments which force quasirandomness in regular tournaments?
\end{ques}

More ambitiously, we ask the following.

\begin{ques}
\label{ques:almostAll}
Is it true that almost every tournament forces quasirandomness in regular tournaments? 
\end{ques}

\begin{ack}
The authors thank Bernard Lidick\'y, Daniel Kr\'a\v{l}, Florian Pfender and Jan Volec for valuable discussions. In particular, we thank them for finding the example in the proof of Proposition~\ref{prop:special}---which disproved a conjecture in an earlier arxiv preprint of this paper---and letting us include it here. 
\end{ack}

\appendix

\section{Flag Algebra Coefficients}
\label{app:coeffs}

The purpose of this appendix is to list all of the coefficients $b_2(F_i^1,F_j^1;J)$ for $1\leq i,j\leq 4$ and $b_3(F_i^q,F_j^q;J)$ for $1\leq i,j\leq 8$ and $q\in\{2,3\}$, where $J$ is a tournament on 5 vertices and the flags $F_i^q$ are as in Section~\ref{sec:flags}. All of these coefficients were computed by computer but can be easily checked by hand. It will be convenient to record these coefficients in matrices. For every such $J$, let $B_2^1(J)$ be the $4\times 4$ matrix in which the entry on the $i$th row and $j$th column is $b_2(F_i^1,F_j^1;J)$. Similarly, for every such $J$ and $q\in \{2,3\}$, let $B_3^q(J)$ be the $8\times 8$ matrix in which the entry on the $i$th row and $j$th column is $b_3(F_i^q,F_j^q;J)$. The matrices are as follows.

For $J=H_{8}$, we have
\[B_2^1(H_{8}) = \frac{1}{120}\begin{bmatrix}
10 & 5 & 0 & 5\\
5 & 10 & 0 & 5\\
0 & 0 & 0 & 0\\
5 & 5 & 0 & 10
\end{bmatrix}\]
\[B_3^2(H_{8}) = \frac{1}{120}\begin{bmatrix}
2 & 1 & 0 & 1 & 0 & 0 & 0 & 1\\
1 & 2 & 0 & 1 & 0 & 0 & 0 & 1\\
0 & 0 & 0 & 0 & 0 & 0 & 0 & 0\\
1 & 1 & 0 & 2 & 0 & 0 & 0 & 1\\
0 & 0 & 0 & 0 & 0 & 0 & 0 & 0\\
0 & 0 & 0 & 0 & 0 & 0 & 0 & 0\\
0 & 0 & 0 & 0 & 0 & 0 & 0 & 0\\
1 & 1 & 0 & 1 & 0 & 0 & 0 & 2
\end{bmatrix}\]
\[B_3^3(H_{8}) = \frac{1}{120}\begin{bmatrix}
0 & 0 & 0 & 0 & 0 & 0 & 0 & 0\\
0 & 0 & 0 & 0 & 0 & 0 & 0 & 0\\
0 & 0 & 0 & 0 & 0 & 0 & 0 & 0\\
0 & 0 & 0 & 0 & 0 & 0 & 0 & 0\\
0 & 0 & 0 & 0 & 0 & 0 & 0 & 0\\
0 & 0 & 0 & 0 & 0 & 0 & 0 & 0\\
0 & 0 & 0 & 0 & 0 & 0 & 0 & 0\\
0 & 0 & 0 & 0 & 0 & 0 & 0 & 0
\end{bmatrix}\]
For $J=H_{9}$, we have
\[B_2^1(H_{9}) = \frac{1}{120}\begin{bmatrix}
6 & 9 & 0 & 3\\
9 & 6 & 0 & 3\\
0 & 0 & 0 & 6\\
3 & 3 & 6 & 6
\end{bmatrix}\]
\[B_3^2(H_{9}) =\frac{1}{120}\begin{bmatrix}
0 & 3 & 0 & 0 & 0 & 0 & 0 & 0\\
3 & 0 & 0 & 0 & 0 & 0 & 0 & 0\\
0 & 0 & 0 & 3 & 0 & 0 & 0 & 3\\
0 & 0 & 3 & 0 & 0 & 0 & 0 & 0\\
0 & 0 & 0 & 0 & 0 & 0 & 0 & 0\\
0 & 0 & 0 & 0 & 0 & 0 & 0 & 0\\
0 & 0 & 0 & 0 & 0 & 0 & 0 & 0\\
0 & 0 & 3 & 0 & 0 & 0 & 0 & 0
\end{bmatrix}\]
\[B_3^3(H_{9}) = \frac{1}{120}\begin{bmatrix}
0 & 0 & 0 & 0 & 0 & 0 & 0 & 0\\
0 & 0 & 0 & 0 & 0 & 0 & 0 & 0\\
0 & 0 & 0 & 0 & 0 & 0 & 0 & 0\\
0 & 0 & 0 & 0 & 0 & 0 & 0 & 0\\
0 & 0 & 0 & 0 & 0 & 0 & 0 & 0\\
0 & 0 & 0 & 0 & 0 & 0 & 0 & 0\\
0 & 0 & 0 & 0 & 0 & 0 & 0 & 0\\
0 & 0 & 0 & 0 & 0 & 0 & 0 & 6
\end{bmatrix}\]
For $J=H_{10}$, we have
\[B_2^1(H_{10}) = \frac{1}{120}\begin{bmatrix}
4 & 8 & 1 & 3\\
8 & 4 & 2 & 2\\
1 & 2 & 2 & 7\\
3 & 2 & 7 & 4
\end{bmatrix}\]
\[B_3^2(H_{10}) = \frac{1}{120}\begin{bmatrix}
0 & 0 & 1 & 1 & 0 & 0 & 0 & 0\\
0 & 0 & 2 & 0 & 0 & 0 & 0 & 0\\
1 & 2 & 2 & 1 & 0 & 0 & 0 & 0\\
1 & 0 & 1 & 0 & 0 & 0 & 0 & 0\\
0 & 0 & 0 & 0 & 0 & 0 & 0 & 1\\
0 & 0 & 0 & 0 & 0 & 0 & 0 & 0\\
0 & 0 & 0 & 0 & 0 & 0 & 0 & 1\\
0 & 0 & 0 & 0 & 1 & 0 & 1 & 0
\end{bmatrix}\]
\[B_3^3(H_{10}) = \frac{1}{120}\begin{bmatrix}
0 & 0 & 0 & 0 & 0 & 0 & 0 & 0\\
0 & 0 & 0 & 0 & 0 & 0 & 0 & 1\\
0 & 0 & 0 & 0 & 0 & 0 & 0 & 1\\
0 & 0 & 0 & 0 & 0 & 0 & 0 & 1\\
0 & 0 & 0 & 0 & 0 & 0 & 0 & 1\\
0 & 0 & 0 & 0 & 0 & 0 & 0 & 1\\
0 & 0 & 0 & 0 & 0 & 0 & 0 & 1\\
0 & 1 & 1 & 1 & 1 & 1 & 1 & 0
\end{bmatrix}\]
For $J=H_{11}$, we have
\[B_2^1(H_{11}) = \frac{1}{120}\begin{bmatrix}
6 & 6 & 3 & 3\\
6 & 6 & 0 & 6\\
3 & 0 & 0 & 3\\
3 & 6 & 3 & 6
\end{bmatrix}\]
\[B_3^2(H_{11}) = \frac{1}{120}\begin{bmatrix}
0 & 0 & 3 & 0 & 0 & 0 & 0 & 0\\
0 & 0 & 0 & 3 & 0 & 0 & 0 & 0\\
3 & 0 & 0 & 0 & 0 & 0 & 0 & 0\\
0 & 3 & 0 & 0 & 0 & 0 & 0 & 0\\
0 & 0 & 0 & 0 & 0 & 0 & 0 & 0\\
0 & 0 & 0 & 0 & 0 & 0 & 0 & 3\\
0 & 0 & 0 & 0 & 0 & 0 & 0 & 0\\
0 & 0 & 0 & 0 & 0 & 3 & 0 & 0
\end{bmatrix}\]
\[B_3^3(H_{11}) = \frac{1}{120}\begin{bmatrix}
0 & 0 & 0 & 0 & 0 & 0 & 0 & 3\\
0 & 0 & 0 & 0 & 0 & 0 & 0 & 0\\
0 & 0 & 0 & 0 & 0 & 0 & 0 & 0\\
0 & 0 & 0 & 0 & 0 & 0 & 0 & 0\\
0 & 0 & 0 & 0 & 0 & 0 & 0 & 0\\
0 & 0 & 0 & 0 & 0 & 0 & 0 & 0\\
0 & 0 & 0 & 0 & 0 & 0 & 0 & 0\\
3 & 0 & 0 & 0 & 0 & 0 & 0 & 0
\end{bmatrix}\]
For $J=H_{12}$, we have
\[B_2^1(H_{12}) = \frac{1}{120}\begin{bmatrix}
0 & 3 & 6 & 3\\
3 & 0 & 6 & 3\\
6 & 6 & 6 & 6\\
3 & 3 & 6 & 0
\end{bmatrix}\]
\[B_3^2(H_{12}) = \frac{1}{120}\begin{bmatrix}
0 & 0 & 0 & 0 & 0 & 0 & 0 & 0\\
0 & 0 & 0 & 0 & 0 & 0 & 0 & 0\\
0 & 0 & 0 & 0 & 3 & 0 & 0 & 0\\
0 & 0 & 0 & 0 & 0 & 0 & 0 & 0\\
0 & 0 & 3 & 0 & 0 & 0 & 0 & 0\\
0 & 0 & 0 & 0 & 0 & 0 & 3 & 0\\
0 & 0 & 0 & 0 & 0 & 3 & 0 & 0\\
0 & 0 & 0 & 0 & 0 & 0 & 0 & 0
\end{bmatrix}\]
\[B_3^3(H_{12}) = \frac{1}{120}\begin{bmatrix}
0 & 0 & 0 & 0 & 0 & 0 & 0 & 3\\
0 & 0 & 0 & 0 & 0 & 3 & 0 & 0\\
0 & 0 & 0 & 3 & 0 & 0 & 0 & 0\\
0 & 0 & 3 & 0 & 0 & 0 & 0 & 0\\
0 & 0 & 0 & 0 & 0 & 0 & 3 & 0\\
0 & 3 & 0 & 0 & 0 & 0 & 0 & 0\\
0 & 0 & 0 & 0 & 3 & 0 & 0 & 0\\
3 & 0 & 0 & 0 & 0 & 0 & 0 & 0
\end{bmatrix}\]
For $J=H_{13}$, we have
\[B_2^1(H_{13}) = \frac{1}{120}\begin{bmatrix}
2 & 1 & 4 & 5\\
1 & 2 & 8 & 1\\
4 & 8 & 8 & 4\\
5 & 1 & 4 & 2
\end{bmatrix}\]
\[B_3^2(H_{13}) = \frac{1}{120}\begin{bmatrix}
0 & 0 & 0 & 0 & 0 & 0 & 1 & 0\\
0 & 0 & 0 & 0 & 0 & 0 & 1 & 0\\
0 & 0 & 0 & 0 & 0 & 0 & 0 & 0\\
0 & 0 & 0 & 0 & 1 & 0 & 0 & 0\\
0 & 0 & 0 & 1 & 2 & 0 & 0 & 1\\
0 & 0 & 0 & 0 & 0 & 0 & 0 & 0\\
1 & 1 & 0 & 0 & 0 & 0 & 2 & 0\\
0 & 0 & 0 & 0 & 1 & 0 & 0 & 0
\end{bmatrix}\]
\[B_3^3(H_{13}) = \frac{1}{120}\begin{bmatrix}
0 & 0 & 0 & 0 & 0 & 0 & 0 & 0\\
0 & 0 & 1 & 1 & 1 & 0 & 1 & 0\\
0 & 1 & 0 & 0 & 1 & 1 & 1 & 0\\
0 & 1 & 0 & 0 & 1 & 1 & 1 & 0\\
0 & 1 & 1 & 1 & 0 & 1 & 0 & 0\\
0 & 0 & 1 & 1 & 1 & 0 & 1 & 0\\
0 & 1 & 1 & 1 & 0 & 1 & 0 & 0\\
0 & 0 & 0 & 0 & 0 & 0 & 0 & 0
\end{bmatrix}\]
For $J=H_{14}$, we have
\[B_2^1(H_{14}) = \frac{1}{120}\begin{bmatrix}
4 & 2 & 3 & 5\\
2 & 4 & 6 & 2\\
3 & 6 & 6 & 3\\
5 & 2 & 3 & 4
\end{bmatrix}\]
\[B_3^2(H_{14}) = \frac{1}{120}\begin{bmatrix}
0 & 0 & 0 & 0 & 1 & 0 & 0 & 1\\
0 & 0 & 0 & 0 & 1 & 0 & 1 & 0\\
0 & 0 & 0 & 0 & 0 & 0 & 0 & 0\\
0 & 0 & 0 & 0 & 1 & 0 & 1 & 0\\
1 & 1 & 0 & 1 & 0 & 0 & 0 & 0\\
0 & 0 & 0 & 0 & 0 & 0 & 0 & 0\\
0 & 1 & 0 & 1 & 0 & 0 & 0 & 1\\
1 & 0 & 0 & 0 & 0 & 0 & 1 & 0
\end{bmatrix}\]
\[B_3^3(H_{14}) = \frac{1}{120}\begin{bmatrix}
0 & 0 & 0 & 0 & 0 & 0 & 0 & 0\\
0 & 2 & 0 & 0 & 0 & 1 & 0 & 0\\
0 & 0 & 2 & 1 & 0 & 0 & 0 & 0\\
0 & 0 & 1 & 2 & 0 & 0 & 0 & 0\\
0 & 0 & 0 & 0 & 2 & 0 & 1 & 0\\
0 & 1 & 0 & 0 & 0 & 2 & 0 & 0\\
0 & 0 & 0 & 0 & 1 & 0 & 2 & 0\\
0 & 0 & 0 & 0 & 0 & 0 & 0 & 0
\end{bmatrix}\]
For $J=H_{15}$, we have
\[B_2^1(H_{15}) = \frac{1}{120}\begin{bmatrix}
4 & 2 & 7 & 3\\
2 & 4 & 2 & 8\\
7 & 2 & 2 & 1\\
3 & 8 & 1 & 4
\end{bmatrix}\]
\[B_3^2(H_{15}) = \frac{1}{120}\begin{bmatrix}
0 & 0 & 0 & 0 & 1 & 0 & 1 & 0\\
0 & 0 & 0 & 0 & 0 & 1 & 0 & 1\\
0 & 0 & 0 & 0 & 0 & 0 & 0 & 0\\
0 & 0 & 0 & 0 & 0 & 2 & 0 & 0\\
1 & 0 & 0 & 0 & 0 & 0 & 0 & 0\\
0 & 1 & 0 & 2 & 0 & 2 & 0 & 1\\
1 & 0 & 0 & 0 & 0 & 0 & 0 & 0\\
0 & 1 & 0 & 0 & 0 & 1 & 0 & 0
\end{bmatrix}\]
\[B_3^3(H_{15}) = \frac{1}{120}\begin{bmatrix}
0 & 1 & 1 & 1 & 1 & 1 & 1 & 0\\
1 & 0 & 0 & 0 & 0 & 0 & 0 & 0\\
1 & 0 & 0 & 0 & 0 & 0 & 0 & 0\\
1 & 0 & 0 & 0 & 0 & 0 & 0 & 0\\
1 & 0 & 0 & 0 & 0 & 0 & 0 & 0\\
1 & 0 & 0 & 0 & 0 & 0 & 0 & 0\\
1 & 0 & 0 & 0 & 0 & 0 & 0 & 0\\
0 & 0 & 0 & 0 & 0 & 0 & 0 & 0
\end{bmatrix}\]
For $J=H_{16}$, we have
\[B_2^1(H_{16}) = \frac{1}{120}\begin{bmatrix}
6 & 3 & 6 & 3\\
3 & 6 & 0 & 9\\
6 & 0 & 0 & 0\\
3 & 9 & 0 & 6
\end{bmatrix}\]
\[B_3^2(H_{16}) = \frac{1}{120}\begin{bmatrix}
0 & 0 & 0 & 0 & 0 & 3 & 0 & 0\\
0 & 0 & 0 & 0 & 0 & 3 & 0 & 0\\
0 & 0 & 0 & 0 & 0 & 0 & 0 & 0\\
0 & 0 & 0 & 0 & 0 & 0 & 0 & 3\\
0 & 0 & 0 & 0 & 0 & 0 & 0 & 0\\
3 & 3 & 0 & 0 & 0 & 0 & 0 & 0\\
0 & 0 & 0 & 0 & 0 & 0 & 0 & 0\\
0 & 0 & 0 & 3 & 0 & 0 & 0 & 0
\end{bmatrix}\]
\[B_3^3(H_{16}) = \frac{1}{120}\begin{bmatrix}
6 & 0 & 0 & 0 & 0 & 0 & 0 & 0\\
0 & 0 & 0 & 0 & 0 & 0 & 0 & 0\\
0 & 0 & 0 & 0 & 0 & 0 & 0 & 0\\
0 & 0 & 0 & 0 & 0 & 0 & 0 & 0\\
0 & 0 & 0 & 0 & 0 & 0 & 0 & 0\\
0 & 0 & 0 & 0 & 0 & 0 & 0 & 0\\
0 & 0 & 0 & 0 & 0 & 0 & 0 & 0\\
0 & 0 & 0 & 0 & 0 & 0 & 0 & 0
\end{bmatrix}\]
For $J=H_{17}$, we have
\[B_2^1(H_{17}) = \frac{1}{120}\begin{bmatrix}
2 & 4 & 5 & 3\\
4 & 2 & 4 & 4\\
5 & 4 & 4 & 5\\
3 & 4 & 5 & 2
\end{bmatrix}\]
\[B_3^2(H_{17}) = \frac{1}{120}\begin{bmatrix}
0 & 0 & 0 & 0 & 0 & 1 & 0 & 0\\
0 & 0 & 0 & 0 & 1 & 0 & 0 & 0\\
0 & 0 & 0 & 0 & 0 & 1 & 1 & 1\\
0 & 0 & 0 & 0 & 0 & 0 & 1 & 0\\
0 & 1 & 0 & 0 & 0 & 1 & 0 & 0\\
1 & 0 & 1 & 0 & 1 & 0 & 0 & 0\\
0 & 0 & 1 & 1 & 0 & 0 & 0 & 0\\
0 & 0 & 1 & 0 & 0 & 0 & 0 & 0
\end{bmatrix}\]
\[B_3^3(H_{17}) = \frac{1}{120}\begin{bmatrix}
0 & 1 & 1 & 0 & 1 & 0 & 0 & 0\\
1 & 0 & 0 & 1 & 0 & 0 & 0 & 0\\
1 & 0 & 0 & 0 & 0 & 0 & 1 & 0\\
0 & 1 & 0 & 0 & 0 & 0 & 0 & 1\\
1 & 0 & 0 & 0 & 0 & 1 & 0 & 0\\
0 & 0 & 0 & 0 & 1 & 0 & 0 & 1\\
0 & 0 & 1 & 0 & 0 & 0 & 0 & 1\\
0 & 0 & 0 & 1 & 0 & 1 & 1 & 0
\end{bmatrix}\]
For $J=H_{18}$, we have
\[B_2^1(H_{18}) = \frac{1}{120}\begin{bmatrix}
0 & 3 & 6 & 3\\
3 & 0 & 6 & 3\\
6 & 6 & 6 & 6\\
3 & 3 & 6 & 0
\end{bmatrix}\]
\[B_3^2(H_{18}) = \frac{1}{120}\begin{bmatrix}
0 & 0 & 0 & 0 & 0 & 0 & 0 & 0\\
0 & 0 & 0 & 0 & 0 & 0 & 0 & 0\\
0 & 0 & 0 & 0 & 1 & 1 & 1 & 0\\
0 & 0 & 0 & 0 & 0 & 0 & 0 & 0\\
0 & 0 & 1 & 0 & 0 & 1 & 1 & 0\\
0 & 0 & 1 & 0 & 1 & 0 & 1 & 0\\
0 & 0 & 1 & 0 & 1 & 1 & 0 & 0\\
0 & 0 & 0 & 0 & 0 & 0 & 0 & 0
\end{bmatrix}\]
\[B_3^3(H_{18}) = \frac{1}{120}\begin{bmatrix}
0 & 0 & 0 & 1 & 0 & 1 & 1 & 0\\
0 & 0 & 1 & 0 & 1 & 0 & 0 & 1\\
0 & 1 & 0 & 0 & 1 & 0 & 0 & 1\\
1 & 0 & 0 & 0 & 0 & 1 & 1 & 0\\
0 & 1 & 1 & 0 & 0 & 0 & 0 & 1\\
1 & 0 & 0 & 1 & 0 & 0 & 1 & 0\\
1 & 0 & 0 & 1 & 0 & 1 & 0 & 0\\
0 & 1 & 1 & 0 & 1 & 0 & 0 & 0
\end{bmatrix}\]
For $J=H_{19}$, we have
\[B_2^1(H_{19}) = \frac{1}{120}\begin{bmatrix}
0 & 0 & 5 & 5\\
0 & 0 & 10 & 0\\
5 & 10 & 10 & 5\\
5 & 0 & 5 & 0
\end{bmatrix}\]
\[B_3^2(H_{19}) = \frac{1}{120}\begin{bmatrix}
0 & 0 & 0 & 0 & 0 & 0 & 0 & 0\\
0 & 0 & 0 & 0 & 0 & 0 & 0 & 0\\
0 & 0 & 0 & 0 & 0 & 0 & 0 & 0\\
0 & 0 & 0 & 0 & 0 & 0 & 0 & 0\\
0 & 0 & 0 & 0 & 0 & 0 & 5 & 0\\
0 & 0 & 0 & 0 & 0 & 0 & 0 & 0\\
0 & 0 & 0 & 0 & 5 & 0 & 0 & 0\\
0 & 0 & 0 & 0 & 0 & 0 & 0 & 0
\end{bmatrix}\]
\[B_3^3(H_{19}) = \frac{1}{120}\begin{bmatrix}
0 & 0 & 0 & 0 & 0 & 0 & 0 & 0\\
0 & 0 & 0 & 0 & 0 & 0 & 5 & 0\\
0 & 0 & 0 & 0 & 0 & 5 & 0 & 0\\
0 & 0 & 0 & 0 & 5 & 0 & 0 & 0\\
0 & 0 & 0 & 5 & 0 & 0 & 0 & 0\\
0 & 0 & 5 & 0 & 0 & 0 & 0 & 0\\
0 & 5 & 0 & 0 & 0 & 0 & 0 & 0\\
0 & 0 & 0 & 0 & 0 & 0 & 0 & 0
\end{bmatrix}\]


\begin{thebibliography}{99}

\bibitem{AlonSpencer}
N.~Alon and J.~H. Spencer.
\newblock {\em The probabilistic method}.
\newblock Wiley Series in Discrete Mathematics and Optimization. John Wiley \&
  Sons, Inc., Hoboken, NJ, fourth edition, 2016.

\bibitem{BucicLongShapiraSudakov21}
M.~Buci\'{c}, E.~Long, A.~Shapira, and B.~Sudakov.
\newblock Tournament quasirandomness from local counting.
\newblock {\em Combinatorica}, 41(2):175--208, 2021.

\bibitem{BurkeLidickyPfenderPhilips22+}
D.~Burke, B.~Lidick\'y, F.~Pfender, and M.~Philips.
\newblock Inducibility of 4-vertex tournaments.
\newblock E-print arXiv:2103.07047v2, 2022.

\bibitem{ChanGrzesikKralNoel20}
T.~F.~N. Chan, A.~Grzesik, D.~Kr\'{a}l', and J.~A. Noel.
\newblock Cycles of length three and four in tournaments.
\newblock {\em J. Combin. Theory Ser. A}, 175:105276, 23, 2020.

\bibitem{Chan+20}
T.~F.~N. Chan, D.~Kr\'{a}l', J.~A. Noel, Y.~Pehova, M.~Sharifzadeh, and
  J.~Volec.
\newblock Characterization of quasirandom permutations by a pattern sum.
\newblock {\em Random Structures Algorithms}, 57(4):920--939, 2020.

\bibitem{Chung14}
F.~Chung.
\newblock From quasirandom graphs to graph limits and graphlets.
\newblock {\em Adv. in Appl. Math.}, 56:135--174, 2014.

\bibitem{ChungGraham90}
F.~R.~K. Chung and R.~L. Graham.
\newblock Quasi-random hypergraphs.
\newblock {\em Random Structures Algorithms}, 1(1):105--124, 1990.

\bibitem{ChungGraham91a}
F.~R.~K. Chung and R.~L. Graham.
\newblock Quasi-random set systems.
\newblock {\em J. Amer. Math. Soc.}, 4(1):151--196, 1991.

\bibitem{ChungGraham91b}
F.~R.~K. Chung and R.~L. Graham.
\newblock Quasi-random tournaments.
\newblock {\em J. Graph Theory}, 15(2):173--198, 1991.

\bibitem{ChungGrahamWilson89}
F.~R.~K. Chung, R.~L. Graham, and R.~M. Wilson.
\newblock Quasi-random graphs.
\newblock {\em Combinatorica}, 9(4):345--362, 1989.

\bibitem{ConlonFoxSudakov10}
D.~Conlon, J.~Fox, and B.~Sudakov.
\newblock An approximate version of {S}idorenko's conjecture.
\newblock {\em Geom. Funct. Anal.}, 20(6):1354--1366, 2010.

\bibitem{ConlonFoxSudakov18}
D.~Conlon, J.~Fox, and B.~Sudakov.
\newblock Hereditary quasirandomness without regularity.
\newblock {\em Math. Proc. Cambridge Philos. Soc.}, 164(3):385--399, 2018.

\bibitem{ConlonHanPersonSchacht12}
D.~Conlon, H.~H\`an, Y.~Person, and M.~Schacht.
\newblock Weak quasi-randomness for uniform hypergraphs.
\newblock {\em Random Structures Algorithms}, 40(1):1--38, 2012.

\bibitem{Cooper04}
J.~N. Cooper.
\newblock Quasirandom permutations.
\newblock {\em J. Combin. Theory Ser. A}, 106(1):123--143, 2004.

\bibitem{CooperKralLamaisonMohr22}
J.~W. Cooper, D.~Kr\'{a}l', A.~Lamaison, and S.~Mohr.
\newblock Quasirandom {L}atin squares.
\newblock {\em Random Structures Algorithms}, 61(2):298--308, 2022.

\bibitem{CoreglianoParenteSato19}
L~N. Coregliano, R~F. Parente, and C~M. Sato.
\newblock On the maximum density of fixed strongly connected subtournaments.
\newblock {\em Electron. J. Combin.}, 26(1):Paper No. 1.44, 48, 2019.

\bibitem{CoreglianoRazborov17}
L.~N. Coregliano and A.~A. Razborov.
\newblock On the density of transitive tournaments.
\newblock {\em J. Graph Theory}, 85(1):12--21, 2017.

\bibitem{CoreglianoRazborov23}
L.~N. Coregliano and A.~A. Razborov.
\newblock Natural quasirandomness properties.
\newblock {\em Random Structures Algorithms}, 63(3):624--688, 2023.

\bibitem{CrudeleDukesNoel23+}
G.~Crudele, P.~Dukes, and J.~A. Noel.
\newblock Six permutation patterns force quasirandomness.
\newblock E-print arXiv:2303.04776v3. To appear in \emph{Discrete Anal.}, 2023.

\bibitem{DellamonicaRodl11}
D.~Dellamonica, Jr. and V.~R\"{o}dl.
\newblock Hereditary quasirandom properties of hypergraphs.
\newblock {\em Combinatorica}, 31(2):165--182, 2011.

\bibitem{Gowers06}
W.~T. Gowers.
\newblock Quasirandomness, counting and regularity for 3-uniform hypergraphs.
\newblock {\em Combin. Probab. Comput.}, 15(1-2):143--184, 2006.

\bibitem{Gowers08}
W.~T. Gowers.
\newblock Quasirandom groups.
\newblock {\em Combin. Probab. Comput.}, 17(3):363--387, 2008.

\bibitem{Griffiths13}
S.~Griffiths.
\newblock Quasi-random oriented graphs.
\newblock {\em J. Graph Theory}, 74(2):198--209, 2013.

\bibitem{Grzesik+23}
A.~Grzesik, D.~Il'kovi\v{c}, B.~Kielak, and D.~Kr\'{a}l'.
\newblock Quasirandom-forcing orientations of cycles.
\newblock {\em SIAM J. Discrete Math.}, 37(4):2689--2716, 2023.

\bibitem{GrzesikKralLovaszVolec23}
A.~Grzesik, D.~Kr\'{a}l', L.~M. Lov\'{a}sz, and J.~Volec.
\newblock Cycles of a given length in tournaments.
\newblock {\em J. Combin. Theory Ser. B}, 158:117--145, 2023.

\bibitem{GrzesikKralPikhurko24}
A.~Grzesik, D.~Kr\'{a}l', and O.~Pikhurko.
\newblock Forcing generalised quasirandom graphs efficiently.
\newblock {\em Combin. Probab. Comput.}, 33(1):16--31, 2024.

\bibitem{Hancock+23}
R.~Hancock, A.~Kabela, D.~Kr\'{a}l', T.~Martins, R.~Parente, F.~Skerman, and
  J.~Volec.
\newblock No additional tournaments are quasirandom-forcing.
\newblock {\em European J. Combin.}, 108:Paper No. 103632, 10, 2023.

\bibitem{KalyanasundaramShapira13}
S.~Kalyanasundaram and A.~Shapira.
\newblock A note on even cycles and quasirandom tournaments.
\newblock {\em J. Graph Theory}, 73(3):260--266, 2013.

\bibitem{KendallBabingtonSmith40}
M.~G. Kendall and B.~Babington~Smith.
\newblock On the method of paired comparisons.
\newblock {\em Biometrika}, 31:324--345, 1940.

\bibitem{KohayakawaRodlSkokan02}
Y.~Kohayakawa, V.~R\"{o}dl, and J.~Skokan.
\newblock Hypergraphs, quasi-randomness, and conditions for regularity.
\newblock {\em J. Combin. Theory Ser. A}, 97(2):307--352, 2002.

\bibitem{KralLeeNoel24+}
D.~Kr\'{a}l', J.-B. Lee, and J.~A. Noel.
\newblock Forcing quasirandomness with 4-point permutations.
\newblock E-print arXiv:2407.06869v1, 2024.

\bibitem{KralPikhurko13}
D.~Kr\'{a}l' and O.~Pikhurko.
\newblock Quasirandom permutations are characterized by 4-point densities.
\newblock {\em Geom. Funct. Anal.}, 23(2):570--579, 2013.

\bibitem{Kurecka22}
M.~Kure\v{c}ka.
\newblock Lower bound on the size of a quasirandom forcing set of permutations.
\newblock {\em Combin. Probab. Comput.}, 31(2):304--319, 2022.

\bibitem{LinialMorgenstern16}
N.~Linial and A.~Morgenstern.
\newblock On the number of 4-cycles in a tournament.
\newblock {\em J. Graph Theory}, 83(3):266--276, 2016.

\bibitem{Lovasz79}
L.~Lov\'{a}sz.
\newblock {\em Combinatorial problems and exercises}.
\newblock North-Holland Publishing Co., Amsterdam-New York, 1979.

\bibitem{LovaszSos08}
L.~Lov\'{a}sz and V.~T. S\'{o}s.
\newblock Generalized quasirandom graphs.
\newblock {\em J. Combin. Theory Ser. B}, 98(1):146--163, 2008.

\bibitem{LovaszSzegedy06}
L.~Lov\'{a}sz and B.~Szegedy.
\newblock Limits of dense graph sequences.
\newblock {\em J. Combin. Theory Ser. B}, 96(6):933--957, 2006.

\bibitem{Lovasz12}
L\'{a}szl\'{o} Lov\'{a}sz.
\newblock {\em Large networks and graph limits}, volume~60 of {\em American
  Mathematical Society Colloquium Publications}.
\newblock American Mathematical Society, Providence, RI, 2012.

\bibitem{MaTang22}
J.~Ma and T.~Tang.
\newblock Minimizing cycles in tournaments and normalized {$q$}-norms.
\newblock {\em Comb. Theory}, 2(3):Paper No. 6, 19, 2022.

\bibitem{MossNoel23+}
E.~Moss and J.~A. Noel.
\newblock Off-diagonal {R}amsey multiplicity.
\newblock E-print arXiv:2306.17388v1, 2023.

\bibitem{Razborov07}
A.~A. Razborov.
\newblock Flag algebras.
\newblock {\em J. Symbolic Logic}, 72(4):1239--1282, 2007.

\bibitem{Rodl86}
V.~R\"{o}dl.
\newblock On universality of graphs with uniformly distributed edges.
\newblock {\em Discrete Math.}, 59(1-2):125--134, 1986.

\bibitem{SahSawhney23+}
A.~Sah and M.~Sawhney.
\newblock The intransitive dice kernel: $\frac{1_{x\geq y}-1_{x\leq
  y}}{4}-\frac{3(x-y)(1+xy)}{8}$.
\newblock E-print arXiv:2302.11293v1, 2023.

\bibitem{SkokanThoma04}
J.~Skokan and L.~Thoma.
\newblock Bipartite subgraphs and quasi-randomness.
\newblock {\em Graphs Combin.}, 20(2):255--262, 2004.

\bibitem{Thomason87a}
A.~Thomason.
\newblock Pseudorandom graphs.
\newblock In {\em Random graphs '85 ({P}ozna\'{n}, 1985)}, volume 144 of {\em
  North-Holland Math. Stud.}, pages 307--331. North-Holland, Amsterdam, 1987.

\bibitem{Thomason87b}
A.~Thomason.
\newblock Random graphs, strongly regular graphs and pseudorandom graphs.
\newblock In {\em Surveys in combinatorics 1987 ({N}ew {C}ross, 1987)}, volume
  123 of {\em London Math. Soc. Lecture Note Ser.}, pages 173--195. Cambridge
  Univ. Press, Cambridge, 1987.

\bibitem{Thornblad18}
E.~Th\"{o}rnblad.
\newblock Decomposition of tournament limits.
\newblock {\em European J. Combin.}, 67:96--125, 2018.

\bibitem{ZhaoZhou20}
Y.~Zhao and Y.~Zhou.
\newblock Impartial digraphs.
\newblock {\em Combinatorica}, 40(6):875--896, 2020.

\end{thebibliography}
\end{document}